\documentclass[11pt,reqno]{amsart}
\usepackage[T1]{fontenc}
\usepackage[utf8]{inputenc}
\usepackage{appendix}
\usepackage{paralist}
\usepackage{longtable}
\usepackage{graphicx}
\usepackage{amsmath}
\usepackage{amsfonts}
\usepackage{fullwidth}
\usepackage{amssymb}
\usepackage{upgreek}
\usepackage{marvosym}
\usepackage{amsthm}
\usepackage{caption}
\usepackage{subcaption}
\usepackage{amstext}
\usepackage{array}
\usepackage{euler}
\usepackage{adjustbox}
\usepackage{times}
\usepackage{lscape}
\usepackage{lipsum}
\usepackage{xspace}
\usepackage{color}
\usepackage{ltxtable}
\newcommand{\MATLAB}{\textsc{Matlab}\xspace}
\usepackage{tabularx,ragged2e,booktabs,caption}
\newcolumntype{C}[1]{>{\Centering}m{#1}}

\usepackage{accents}
\newcommand{\ubar}[1]{\underaccent{\bar}{#1}}

\usepackage[margin=1in]{geometry}
\numberwithin{equation}{section}
\usepackage{algorithm}
\usepackage{algorithmic}
\makeatletter
\def\listofalgorithms{\@starttoc{loa}\listalgorithmname}
\def\l@algorithm{\@tocline{0}{3pt plus2pt}{0pt}{1.9em}{}}
\renewcommand{\ALG@name}{Algorithm}
\renewcommand{\listalgorithmname}{List of \ALG@name s}
\numberwithin{algorithm}{section}
\theoremstyle{definition}

\theoremstyle{remark}
\newtheorem{rem}[algorithm]{Remark}
\theoremstyle{theorem}
\newtheorem{thm}[algorithm]{Theorem}
\theoremstyle{proposition}

\theoremstyle{example}
\newtheorem{example}[algorithm]{Example}

\theoremstyle{corollary}

\makeatother

\newcommand{\R}{\mathbb{R}}

\usepackage{scalerel}

\usepackage{hyperref}
\hypersetup{
    colorlinks=true,
    linkcolor=blue,
    filecolor=magenta,
    urlcolor=cyan,
}
\usepackage[table]{xcolor}
\newcommand{\af}[1]{{\color{blue}#1}}

\DeclareMathOperator{\argmin}{argmin}
\usepackage{soul}
\usepackage[normalem]{ulem}
\begin{document}
\title[Semismooth Newton method for bilevel optimization]{Semismooth Newton-type method
for bilevel optimization: Global convergence and extensive numerical experiments}
\author[]{Andreas Fischer$^{\dag}$}

\author[]{Alain B. Zemkoho$^{\ddag}$}
\author[]{Shenglong Zhou$^{\ddag}$\\\\
$^{\dag}$D\MakeLowercase{epartment of} M\MakeLowercase{athematics}, T\MakeLowercase{echnical} U\MakeLowercase{niversity of} D\MakeLowercase{resden}, G\MakeLowercase{ermany}\\
\MakeLowercase{\textsf{andreas.fischer@tu-dresden.de}}\\\\
$^{\ddag}$S\MakeLowercase{chool of} M\MakeLowercase{athematics}, U\MakeLowercase{niversity of} S\MakeLowercase{outhampton}, U\MakeLowercase{nited} K\MakeLowercase{ingdom}\\
\MakeLowercase{$\textsf{\{}$\textsf{a.b.zemkoho, shenglong.zhou}$\textsf{\}}$\textsf{@soton.ac.uk}}}%

\thanks{The work of the first author is funded by the Volkswagen Foundation and the second and third authors are funded by EPSRC Grant {EP/P022553/1}.}
\subjclass[2010]{90C26, 90C30, 90C46, 90C53}%
\keywords{Bilevel optimization, Newton method}%

\date{\today}

\begin{abstract}
We consider the standard optimistic bilevel optimization problem, in particular upper- and lower-level constraints can be coupled. By means of the lower-level value function, the problem is transformed into a single-level optimization problem with a penalization of the value function constraint.
For treating the latter problem, we develop a framework that does not rely on the direct computation of the lower-level value function or its derivatives. For each penalty parameter, the framework leads to a semismooth system of equations. This allows us to extend the semismooth Newton method to bilevel optimization. Besides global convergence properties of the method, we focus on achieving local superlinear convergence to a solution of the semismooth system. To this end, we formulate an appropriate CD-regularity assumption and derive suffcient conditions so that it is fulfilled. Moreover, we develop conditions to guarantee that a solution of the semismooth system is a local solution of the bilevel optimization problem. Extensive numerical experiments on $124$ examples of nonlinear bilevel optimization problems from the literature show that this approach exhibits a remarkable performance, where only a few penalty parameters need to be considered.
\end{abstract}
\maketitle

\section{Introduction}\label{Introduction}

We consider the standard optimistic bilevel optimization problem
\begin{equation}\label{P}
   \underset{x,y}\min~F(x,y)\quad\mbox{s.t.}\quad G(x,y)\le 0,\;\, y\in S(x),
\end{equation}
also known as the upper-level problem, where the set-valued mapping $S:\R^n\rightrightarrows\R^m$ describes the optimal
solution set of the lower-level problem
\begin{equation}\label{lower-level problem-0}
\min_z f(x,z)\quad\mbox{s.t.}\quad g(x,z)\le 0,
\end{equation}
i.e.,
\begin{equation}\label{lower-level problem}
    S(x):= 
    \underset{z}\argmin\left\{f(x,z)\mid g(x,z)\le 0\right\}.
\end{equation}
Throughout the paper, the functions $F,\,f:\R^n\times\R^m\to\R$, $G:\R^n\times\R^m\to\R^p$, and $g:\R^n\times\R^m\to\R^q$
are assumed to be twice continuously differentiable. As usual, we call $F$ (resp. $f$) the upper-level (resp. lower-level) objective function, whereas $G$ and $g$ are called upper-level and lower-level constraint functions, respectively. Finally, $x$ and $y$ represent upper-level and lower-level variables.

In order to derive optimality conditions for the bilevel optimization problem or to treat it
numerically, two main approaches for reformulating \eqref{P} as a single-level problem exist.
One approach is to replace
the lower-level problem by its Karush-Kuhn-Tucker (KKT) conditions. This leads to a mathematical
program with complementarity constraints (MPCC). However, quite strong assumptions are needed to show
that a local (or global) optimal solution of the MPCC yields a local (or global) optimal solution of the corresponding bilevel optimization problem; for details, the reader is referred to \cite{DempeDuttaBlpMpec2010}.
Therefore, we do not pursue this approach here. Nevertheless, we would like to underline that solving an
MPCC is challenging, and only a few Newton-type methods with global or local fast convergence exist
\cite{IzmailovPogosyanSolodov2012,LiHuangJian2015}. Even if we disregard the discrepancies between
(local or global) solutions of the bilevel problem and the MPCC for a moment, it is an open question
which conditions in the context of bilevel problems are implied by assumptions needed for the
local convergence analysis of a Newton-type method for the corresponding KKT reformulation.

The second main approach to transform a bilevel program into a single-level optimization problem is the lower-level value function (LLVF) reformulation \cite{Outrata1990,YeZhuOptCondForBilevel1995}, which provides the basis for the developments in this paper. More in detail, the LLVF
$\varphi:\R^n\to\R$ is given by
\begin{equation}\label{varphi}
   \varphi(x) := \min_z\left\{f(x,z)\mid g(x,z)\le 0\right\},
\end{equation}
where, for the sake of simplicity, we assume throughout the paper that $S(x)\neq\emptyset$ for all $x\in \mathbb{R}^n$ so that $\varphi$ is indeed finite-valued on $\R^n$.
Then, w.r.t.  local and global minimizers, the bilevel program \eqref{P} is equivalent
to the optimization problem
\begin{equation}\label{LLVF}
    \underset{x,y}\min~F(x,y)\;\; \mbox{s.t.}\;\; G(x,y)\le 0,\;\; g(x,y)\le 0,\;\; f(x,y)\le\varphi(x).
\end{equation}
In general, this is a nonconvex constrained optimization problem containing the typically
nondifferentiable LLVF $\varphi$. Even if all the functions involved in \eqref{P} are fully convex (i.e., convex w.r.t. $(x,y)$),
the feasible set of problem \eqref{LLVF} is generally nonconvex.

Several algorithms for computing a stationary point of the LLVF reformulation were already
designed and analyzed. For the case where the feasible set of the lower-level problem does not depend on the upper-level variable $x$
such methods can be found in \cite{LinXuYeOnSolving2014,XuYeASmoothing2014,XuYeZhang2015Smoothing}.
The algorithms in \cite{DempeFranke2016OntheConvex, DempeFrankeSolution2014} were suggested
to solve special cases of problem \eqref{LLVF}, where relaxation schemes are used to deal with the
value function \eqref{varphi}. Finally, the authors of \cite{LamparielloSagratella2017}
proposed numerical methods to solve special bilevel programs by exploiting a connection between problem \eqref{LLVF}
and a generalized Nash equilibrium problem. 
In addition, we note that the LLVF reformulation has been used recently for methods for certain classes of
mixed integer bilevel programming problems, see, e.g.,
\cite{FischettiEtal2017ANewGeneral,LozanoSmith2017AValueFunctionBased,RalphsEtal2017ABranch}.

In this paper, we develop a framework for solving problem \eqref{LLVF}, which does not rely on the direct computation of the LLVF \eqref{varphi} or its derivatives, as it is the case in \cite{LinXuYeOnSolving2014,XuYeASmoothing2014,XuYeZhang2015Smoothing}. Thanks to this framework, we are able to extend the semismooth Newton method to bilevel optimization, for the first time in the literature. The ingredients used to construct the framework and to establish global convergence of the method are well-known in the literature. At first, we use partial calmness \cite{YeZhuOptCondForBilevel1995} to move the value function constraint, i.e., $f(x,y)\le\varphi(x)$, to the upper-level objective function, by means of partial exact penalization. For the resulting problem, necessary KKT-type optimality conditions are derived and reformulated as a square nonsmooth semismooth system of equations which depends on the penalty parameter.

For the aforementioned system of equations, global convergence of a semismooth Newton algorithm is established based on \cite{DeLuca1996}, see also \cite{Qi1993convergence,QiJiangSemismooth1997,QiSun1999} for related results. Obtaining local superlinear convergence of a semismooth Newton algorithm usually relies on the nonsingularity of some generalized Jacobian of the system of equations, see
\cite{FischerASpecial1992} and more general results in {\cite{FFK1998,Kummer1988,QiSunANonsmoothVersion1993}.
For our setting, we will derive conditions that ensure an appropriate nonsingularity property. To this end, upper- and lower-level linear independence conditions as well as a certain SSOSC (strong second order sufficient condition)-type condition will come into play.} In standard nonlinear optimization, a similar setup guaranties that the reference point, where these conditions are satisfied, also corresponds to a locally optimal solution \cite{Robinson1982}; this is known as Robinson condition. We will derive a Robinson-type condition which guaranties that the reference point corresponds to a locally optimal solution for the penalized bilevel program. For this, the LLVF \eqref{varphi} is required to be second order directionally differentiable in the sense of \cite{Ben-Tal1982, Shapiro1988}.


For the algorithm studied in this paper, we have conducted detailed numerical experiments using the BOLIB (Bilevel Optimization LIBrary) \cite{BOLIB2017} made of $124$ examples of nonlinear bilevel optimization problems from the literature. The true optimal solutions are known for 70\% of these problems. We were able to recover these solutions using a selection of just 9 values of the involved penalization parameter. However, it is important to emphasize that the primary goal of the method is to compute stationary points based on problem \eqref{LLVF}. For each of the 9 values of the penalization parameter, the method converges for at least 87\% of the problems. The algorithm also exhibits a good experimental order of convergence (at least greater or equal to 1.5 for at least about 70\% of the problems) for different values of the parameters.  To the best of our knowledge, this is the first time where an algorithm is proposed for nonlinear bilevel optimization, with computational experiments at such a scale and such a level of success. 

The paper is organized as follows. In the next section, we recall some basic notions and properties, centered around the generalized first and second order differentiation of the LLVF \eqref{varphi}.  Section \ref{Necessary conditions for optimality} discusses the stationarity concept
that will be the basis for the semismoth system suggested later on. There, we recall some basic tools
and the general framework for deriving optimality conditions for the LLVF reformulation \eqref{LLVF}.
In Section \ref{The algorithm}, based on the reformulation in {\cite{FischerASpecial1992}, we
suggest a semismooth system of equations to rewrite the stationarity conditions and establish
a semismooth Newton method for computing stationary points.} In Section \ref{CD-regularity}, we derive sufficient conditions for the CD-regularity
of this semismooth system at a solution in order to guarantee superlinear or quadratic convergence.
%
Finally, Section \ref{Numerical experiments} presents results of a numerical study of the semismooth Newton method on the test problems in the BOLIB library \cite{BOLIB2017}.

\section{Preliminaries}
\label{prelim}

We first introduce some basic notation that will be used throughout. Depending on the
situation, both $y$ and $z$ will be used as lower-level variables.
According to this, for some $(\bar x,\bar y)$ and $(\bar x,\bar z)$, the index sets
\begin{equation}\label{I1u}
   I^1:= I^G(\bar x, \bar y):=  \left\{i\mid G_i(\bar x, \bar y)=0\right\}
\end{equation}
of the active upper-level constraints and
\begin{equation}\label{I1l}
  I^2:= I^g(\bar x, \bar y):= \left\{j\mid g_j(\bar x, \bar y)=0\right\}\quad\mbox{and}\quad
  I^3:= I^g(\bar x, \bar z)
\end{equation}
for the active lower-level constraints are defined.
Moreover, given a multiplier $\bar u\in\R^p$ for the upper-level constraint function $G$ and multipliers $\bar v, \bar w\in\R^q$, each for the lower-level constraint function $g$, then the index sets
\begin{equation}\label{multiplier sets}
\begin{array}{c}
 \eta^1  := \eta^G(\bar x, \bar y, \bar u) := \{i\mid \bar u_i =0, \, G_i(\bar x , \bar y)<0\},\\
\theta^1  :=  \theta^G(\bar x, \bar y, \bar u):= \{i\mid \bar u_i =0, \, G_i(\bar x , \bar y)=0\},\\
\nu^1  :=  \nu^G(\bar x, \bar y, \bar u):= \{i\mid \bar u_i>0, \, G_i(\bar x , \bar y)=0\}
\end{array}
\end{equation}
are defined as subsets of the upper-level indices $\{1,\ldots,p\}$ and, similarly,
\begin{equation}\label{nu2nu3}
    \begin{array}{lllcl}
  \eta^2:= \eta^g(\bar x, \bar y, \bar v),& \theta^2:=\theta^g(\bar x, \bar y, \bar v),&\mbox{and}& \nu^2:=\nu^g(\bar x, \bar y, \bar v),\\
  \eta^3:=\eta^g(\bar x, \bar z, \bar w),& \theta^3:=\theta^g(\bar x, \bar z, \bar w),&\mbox{and}& \nu^3:=\nu^g(\bar x, \bar z, \bar w)
\end{array}
\end{equation}
are defined as subsets of the lower-level indices $\{1,\ldots,q\}$.
The lower-level Lagrangian function $\ell$ defined from  $\R^n\times\R^m\times\R^q$ to $\mathbb{R}$ is given by
 \begin{equation}\label{ell}
    \ell(x,z,w):=f(x,z)+ w^\top g(x,z).
 \end{equation}
For a function depending on upper-level and (or) lower-level variables, we will often use the notations $\nabla_1$ and $\nabla_2$ to denote the gradients of this function w.r.t. the upper-level (resp. lower-level) variable, to avoid any potential confusion. Similarly, $\nabla^2_2$, $\nabla^2_1$, and $\nabla^2_{12}$ could be used when referring to second order derivatives.

Since the optimal value function $\varphi$ \eqref{varphi} is nondifferentiable in general, we  need generalized concepts of differentiability. Let us first recall the usual notion of the directional derivative. For a function $\psi : \mathbb{R}^n \to\mathbb{R}$, its directional derivative at $\bar x \in \mathbb{R}^n$ in the direction $d \in \mathbb{R}^n$ is the limit
\begin{equation}\label{Directional Derivative}
    \psi'(\bar x; d):=\underset{t\downarrow 0}\lim \frac{1}{t}\left[\psi(\bar x + td)-\psi(x)\right],
\end{equation}
provided it exists.
 Differentiable and convex (not necessarily differentiable) functions are examples of directionally differentiable functions \cite{RockafellarConvexAnalBook1970}. The optimal value function $\varphi$ \eqref{varphi} can be directionally differentiable without being differentiable or convex. To underline this, we are going to recall a result from \cite{GauvinDubeau1982}, that will play an important role in this paper.
 To proceed, we will say that the lower-level linear independence constraint qualification (LLICQ) holds at a point $(\bar x, \bar z)$ if the family of vectors
 \begin{equation}\label{LICQ}
 \left\{\nabla_2 g_i(\bar x, \bar z)\mid g_i(\bar x,z)=0\right\}
 \end{equation}
is linearly independent.
For any $(x,z)$, the set of lower-level Lagrange multipliers is given by
\begin{equation}\label{Lambda(x,y)}
\Lambda(x,z): = \big\{w\in \mathbb{R}^q\mid \nabla_2 \ell(x,z,w)=0, \; w\geq 0, \; g(x,z)\leq 0,\; w^{\top} g(x,z)= 0\big\}.
\end{equation}
Also, let us define the set-valued map $K:\R^n\rightrightarrows\R^m$ by
\begin{equation}\label{K(x)}
K(x):=\left\{z\in \mathbb{R}^m|\; g(x,z)\leq 0 \right\},
\end{equation}
which provides all lower-level feasible points for a given value of $x$. 
\begin{thm}{\cite[Corollary 4.4]{GauvinDubeau1982}}\label{Theorem-phi'}
Let  $K$ \eqref{K(x)} be uniformly compact near $\bar x$ (i.e., there is a neighborhood $U$ of $\bar x$ such  that the closure of the set ${\cup}_{x\in U}K(x)$ is compact) and $K(\bar x)\neq \emptyset$. Suppose that LLICQ
\eqref{LICQ} holds at $(\bar x, z)$ for all $z\in S(\bar x)$. Then, $\varphi$ is directionally differentiable at $\bar x$ in any direction $d\in \mathbb{R}^n$ with
\begin{equation}\label{varphi'}
    \varphi'(\bar x; d)= \underset{z\in S(\bar x)}\min\,\nabla_1 \ell(\bar x, z, w_z)^\top d,
\end{equation}
where, for any $z\in S(\bar x)$, $w_z$ is the unique element in $\Lambda(\bar x,z)$.
\end{thm}
Note that the set of lower-level Lagrange multipliers $\Lambda(\bar x,z)$ is a singleton for all $z\in S(\bar x)$
because LLICQ is assumed for all these $z$. If the lower-level problem is convex (i.e., $f$ and $g_i$, $i=1, \ldots, p$ are convex functions),  a version of formula \eqref{varphi'} is given in \cite{golstein2008theory} without requiring uniqueness of lower-level Lagrange multipliers. In the absence of convexity, corresponding formulas of $\varphi'(\bar x; d)$ not requiring single-valuedness of $\Lambda$ are given in \cite{rockafellar1984directional}. However, in the latter case, some second order assumptions are imposed on problem \eqref{lower-level problem-0}. Also, the uniform compactness of $K$ can be replaced by other types of assumptions, i.e., the inner semicontinuity or compactness of the lower-level solution map $S$ \eqref{lower-level problem}; see \cite{MordukhovichBook2006} for details.

As it is shown in \cite{Ben-Tal1982}, one can define a second order directional derivative for
$\psi : \mathbb{R}^n \to \mathbb{R}$ at $\bar x$ in directions $d$ and $e$ by
 \begin{equation}\label{2ndOrderDirectional}
\psi''(\bar x; d, e):=
\lim_{t\downarrow 0} \frac{1}{t^2}\left[\psi\left(\bar x + td + \frac{1}{2}t^2e\right) - \psi(\bar x) -t\psi'(\bar x; d)\right],
 \end{equation}
provided the limit exists.
Next, we recall a formula for the second order directional derivative of the LLVF $\varphi$ \eqref{varphi}
developed in \cite{Shapiro1988}.
For this purpose, a few more assumptions are needed. A point $(\bar x, \bar z, \bar w)$ is said to satisfy the lower-level strict complementarity condition (LSCC)  if
\begin{equation}\label{LSCC}
\theta^3 =\{j\mid \bar w_j=g_j(\bar x,\bar z)=0\}=\emptyset.
\end{equation}
The lower-level submanifold property (LSMP) is said to hold at $(\bar x, d)$ if for all $\bar z$ such that
\begin{equation}\label{S1(x,d)}
\bar z\in S_1(\bar x; d):=\underset{z\in S(\bar x)}\argmin~\nabla_1\ell(\bar x, z, w_z)^\top d,
\end{equation}
the restriction of $S(\bar x)$ on a neighborhood of $\bar z$ is a smooth submanifold of
$
\left\{z\mid g_j(\bar x, z)=0\mbox{ for all }j\in I^3\right\}.
$
Finally, the lower-level second order sufficient condition (LSOSC) is said to hold at $(\bar x, d)$ if
\[
  e^\top \nabla^2_2\ell(\bar x, z, w_z)e >0
\]
for all $z\in S_1(\bar x; d)$ and all $e\in \left[T_{\circ}(z)\right]^{\perp} \setminus \{0\}$ such that $\nabla_2 g_j(\bar x, z)^\top e = 0$ for all $j\in I^3$.
Here $T_{\circ}(z)$ denotes the tangent space \cite{Lee}
of $S(\bar x)$ at $z$. For further details and discussions on these assumptions and related results, including the following one, see  \cite{BonnansShapiroBook2000, Shapiro1988}.
\begin{thm}\label{Theorem-phi''}
Let the assumptions of Theorem \ref{Theorem-phi'} be satisfied and
suppose that LSCC holds at $(\bar x, z, w)$ for all $z\in S_1(\bar x; d)$. If, additionally, the LSMP and LSOSC hold at $(\bar x, d)$, then $\varphi$ \eqref{varphi} is second order directionally differentiable at $\bar x$ in directions $d$ and $e$ with
\begin{equation}\label{varphi"}
    \varphi''(\bar x; d, e)= \underset{z\in S_1(\bar x; d)}\min \left\{\nabla_1\ell(\bar x, z, w_z)^\top e + \xi_d (\bar x, z)\right\},
\end{equation}
with  $S_1(\bar x; d)$ given in \eqref{S1(x,d)} while $\xi_d (\bar x, z)$ is defined by
\begin{equation}\label{varphi"-REST}
  \left\{\begin{array}{l}
       \xi_d (\bar x, z) := \underset{e\in \mathcal{Z}_d(\bar x, z)}\min~(d^\top, e^\top)\nabla^2 \ell(\bar x, z, w_z)(d^\top, e^\top)^\top,\\[1ex]
       \mathcal{Z}_d(\bar x, z) := \left\{e\in \mathbb{R}^m\left|\;\nabla_1 g_i(\bar x, z)^\top d + \nabla_2 g_i(\bar x, z)^\top e =0, \;\, i\in I^3\right.\right\}.
    \end{array}
    \right.
\end{equation}
\end{thm}
Note that $\nabla^2 \ell$ stands for the Hessian w.r.t. the vector made of the first and second variables of the function $\ell$ \eqref{ell}.
We can simplify the assumptions needed in this theorem if we impose that the lower-level problem \eqref{lower-level problem} has a unique optimal solution for $x:=\bar x$; cf. the following theorem from \cite{ShapiroSecondOrder1985}.
\begin{thm}\label{Theorem-phi''0} Let the assumptions of Theorem \ref{Theorem-phi'} be satisfied with $S(\bar x)=\{\bar y\}$ and suppose that
$$
e^\top \nabla^2_2\ell(\bar x, \bar y, \bar w)e >0, \;\;\forall e \neq 0\, \mbox{ s.t. }\;\;
\left\{\begin{array}{ll}
\nabla_2 g_i(\bar x, \bar y)^\top e = 0 & \mbox{for }\; i\in \nu^3,\\
\nabla_2 g_i(\bar x, \bar y)^\top e \leq 0 & \mbox{for }\; i\in \theta^3.
\end{array}\right.
$$
Then, we have the expression
\begin{equation}\label{varphi"01}
    \varphi''(\bar x; d, e)= \nabla_x \ell(\bar x, \bar y, \bar w)^\top e + \xi_{d} (\bar x, \bar y),
\end{equation}
where $\xi_d (\bar x, \bar y)$ is defined as in \eqref{varphi"-REST}, with $z= \bar y$, and $\mathcal{Z}_d(\bar x, \bar y)$ given by
$$
\begin{array}{ll}
  \mathcal{Z}_d(\bar x, \bar y)  :=  \{ e\,| &\nabla_1 g_i(\bar x, \bar y)^\top d + \nabla_2 g_i(\bar x, \bar y)^\top e = 0 \;\mbox{ for } \; i\in \nu^3,\\
                                             & \left. \nabla_1 g_i(\bar x, \bar y)^\top d + \nabla_2 g_i(\bar x, \bar y)^\top e \leq 0 \; \mbox{ for } \; i\in \theta^3 \right\}.
\end{array}
$$
\end{thm}
To extend the concept of the directional derivative to a wider class of functions the notion of a generalized directional derivative was introduced in \cite{ClarkeBook1983} for a function $\psi:\mathbb{R}^n\to \mathbb{R}$ by
 \begin{equation}\label{Clarke Directional Derivative}
    \psi^o(\bar x; d):=\underset{t\downarrow 0}{\underset{x \rightarrow \bar x}\limsup} \frac{1}{t}\left[\psi(x + td)-\psi(x)\right].
\end{equation}
This quantity exists if $\psi$ is Lipschtiz continuous around $\bar x$ \cite[Proposition 2.1.1]{ClarkeBook1983}. Utilizing this notion, Clarke also introduced the  {generalized subdifferential}
\begin{equation}\label{Clarke Subdifferential}
    \partial \psi(\bar x):= \left\{\zeta\in \mathbb{R}^n\mid \psi^o(\bar x; d)\ge\langle\zeta,d\rangle
    \mbox{ for all } d\in \mathbb{R}^n  \right\}.
\end{equation}
Furthermore, if $\psi$ is Lipschitz continuous around $\bar x$, it is is differentiable almost everywhere in
some neighborhood of this point; hence the subdiffential \eqref{Clarke Subdifferential} can also be written as
\begin{equation}\label{Clarke Subdifferential-Derivative}
    \partial \psi(\bar x)= \mbox{co}\partial_B\psi(\bar x)\quad\mbox{with}\quad\partial_B\psi(\bar x):=
    \left\{ \lim~\nabla \psi(x^n)\mid x^n \rightarrow \bar x, \;\, x^n\in D_\psi  \right\},
\end{equation}
where ``co'' stands for the convex hull, $D_\psi$ represents the set of points where $\psi$ is differentiable, and $\partial_B\psi(\bar x)$ is called B-subdifferential of $\psi$ at $\bar x$. For vector valued functions $\psi$, the notions $\partial_B\psi(\bar x)$ and $\partial\psi(\bar x)$ in
\eqref{Clarke Subdifferential-Derivative} remain valid with $\nabla \psi$ denoting the Jacobian of $\psi$ at points where $\psi$ is differentiable. Then, the set $\partial\psi(\bar x)$ is called generalized Jacobian of $\psi$ at $\bar x$.

To complete this section, we briefly review the concept of semismoothness \cite{Mifflin1977}, which is used for the convergence result of the Newton method to be discussed in this paper.
Let a function  $\psi : \mathbb{R}^{n} \rightarrow \mathbb{R}^{m}$ be Lipschitz continuous around $\bar x$. Then,
$\psi$ is called semismooth at $\bar x$ if the limit
\[
\lim \left\{Vd'\mid V\in \partial \psi (\bar x+ td'), \; d' \rightarrow d, \; t\downarrow 0\right\}
\]
exists for all $d\in \mathbb{R}^n$. If, in addition,
\[
Vd - \psi'(\bar x; d) = O(\|d\|^2)
\]
holds for all $V\in \partial \psi (\bar x+d)$ with $d \rightarrow 0$, then $\psi$ is said to be strongly semismooth at $\bar x$.

The function $\psi$ will be said to be  $\psi$ is SC$^1$ if it is continuously differentiable and $\nabla \psi$ is semismooth. Also, $\psi$ is called LC$^2$ function if $\psi$ is twice continuously differentiable and $\nabla^2 \psi$ is locally Lipschitzian.

\section{Necessary conditions for optimality}\label{Necessary conditions for optimality}

The standard approach to derive necessary optimality conditions for the LLVF reformulation \eqref{LLVF} of the bilevel optimization problem \eqref{P} is to consider the  partial penalization
\begin{equation}\label{Penalized-LLVF}
    \underset{x,y}\min~F(x,y) + \lambda (f(x,y) -\varphi(x)) \quad\mbox{s.t.} \quad G(x,y)\leq 0, \; g(x,y)\leq 0,
\end{equation}
where $\lambda\in(0,\infty)$ denotes the penalization parameter.
To deal with the fact that no standard constraint qualification holds for problem \eqref{LLVF} \cite{YeZhuOptCondForBilevel1995}, the authors of the latter paper introduced the partial calmness concept and showed its benefit for obtaining KKT
 conditions for \eqref{LLVF}. More in detail, problem \eqref{LLVF} is said to be partially calm at one of its feasible points $(\bar{x} ,\bar{y} )$ if there exist $\lambda\in(0, \infty)$ and a neighborhood $U$ of $(\bar x,\bar y,0)\in\R^n\times\R^m\times\R$ such that
\begin{equation}\label{partial calmness}
\begin{array}{l}
  F(x,y)-F(\bar{x} , \bar{y} )+\lambda |\varsigma|\geq 0 \\
  \mbox{for all}\;(x,y,\varsigma)\in U\;\mbox{with}\; G(x,y)\leq 0,\, g(x,y)\leq 0, \; f(x,y)- \varphi(x)+\varsigma=0.
\end{array}
\end{equation}
Based on this, the following result can be easily derived.
\begin{thm}\label{equivalen}
 Let  $(\bar x, \bar y)$ be locally optimal for problem \eqref{LLVF}. Then, the latter problem is partially calm at $(\bar x, \bar y)$ if and only if there exists $\lambda \in (0, \infty)$ such that $(\bar x,\bar y)$ is locally optimal for problem \eqref{Penalized-LLVF}.
\end{thm}
Using this theorem, we are now going to establish the necessary optimality conditions that will be the basis of the Newton method in this paper. To proceed, we first need two constraint qualifications. The upper-level Mangasarian-Fromowitz constraint qualification (UMFCQ) will be said to hold at  $(\bar x, \bar y)$ if there exists $d\in \mathbb{R}^{n+m}$ such that
\begin{equation}\label{MFCQ-UL}
   \nabla G_i (\bar x, \bar y)^\top d < 0\quad\mbox{for all}\; i\in I^1 \quad \mbox{and} \quad
   \nabla g_j (\bar x, \bar y)^\top d < 0\quad\mbox{for all}\;j\in I^2.
\end{equation}
The lower-level Mangasarian-Fromowitz constraint qualification (LMFCQ) is satisfied at $\bar x$ if, for any $z\in S(\bar x)$, there exists $d\in \mathbb{R}^m$ such that
\begin{equation}\label{MFCQ-follower}
 \nabla_2 g_j (\bar x,z)^\top d < 0\quad\mbox{for all}\; j\mbox{ with }g_j(\bar x,z)=0.
\end{equation}
Further note that a real-valued function $(x,y) \mapsto \psi(x,y)$ will be said to be fully convex
if it is convex w.r.t. to all variables. The lower-level problem is usually said to be convex if $f$ and $g_j$ ($j=1, \ldots, q$) are convex just w.r.t. the lower-level variable. One can easily check that if a  function $(x,y) \mapsto \psi(x,y)$ is fully convex, then it is convex w.r.t. $x$ and it is convex w.r.t. $y$.
\begin{thm}\label{KN stationarity partial calm}
Let $(\bar x,\bar y)$ be a local optimal solution of problem \eqref{LLVF}, assumed to the partially calm at $(\bar x, \bar y)$. The functions 
$f$ and $g_1,\ldots,g_q$ are assumed to be fully convex. Furthermore, suppose  
that LMFCQ holds at $\bar x$, and that UMFCQ holds at $(\bar x,\bar y)$. Then, there exist $\lambda \in (0,\infty)$, $u\in \mathbb{R}^p$, $(v, w)\in \R^q\times\R^q$, and $z\in \mathbb{R}^m$ such that the following system holds for $(x,y)=(\bar x,\bar y)$:
 \begin{align}
 \nabla_1F(x,y) + \nabla_1 G(x,y)u +\nabla_1 g(x,y)v+\lambda \nabla_1 f(x,y) - \lambda \nabla_1 \ell(x, z, w)&=0,\label{KS-1}\\
\nabla_2 F(x,y) + \nabla_2 G(x,y)u +\nabla_2g(x,y)v +\lambda\nabla_2 f(x,y) &=0, \label{VS-2}\\
\nabla_2f(x, z)+\nabla_2g(x,z)w&=0,\label{KS-2}\\
u\geq 0, \; G(x,y)\leq 0,\; u^{\top}G(x,y)&= 0,\label{VS-3}\\
v\geq 0, \; g(x,y)\leq 0,\; v^{\top} g(x,y)&= 0, \label{VS-4}\\
w\geq 0, \; g(x,z)\leq 0,\; w^{\top} g(x,z)&= 0. \label{KS-3}
 \end{align}
\end{thm}

\begin{proof}Note that similar proof techniques for related results can be found in \cite{DempeDuttaMordukhovichNewNece, DempeZemkohoGenMFCQ, YeZhuOptCondForBilevel1995}.  Since the functions $f$, $g_1,\ldots,g_q$ involved in the lower-level problem are fully convex and sufficiently smooth, the optimal value function $\varphi$ \eqref{varphi} is convex \cite[Lemma 2.1]{TaninoOgawa1984} and, hence, locally Lipschitz continuous around $x$ ($|\varphi(x)|<\infty$ for any $x$ was assumed throughout). Therefore, for any $\lambda>0$, problem \eqref{Penalized-LLVF} is a Lipschitz continuous optimization problem.
Moreover, under the partial calmness condition, Theorem \ref{equivalen} guarantees that $\lambda>0$ existst such that $(\bar x,\bar y)$ is a local optimal solution of problem \eqref{Penalized-LLVF}. Now, applying the necessary optimality conditions for locally Lipschitz optimization problems to \eqref{Penalized-LLVF} and taking into account that UMFCQ \eqref{MFCQ-UL} holds at $(\bar x,\bar y)$, we obtain the existence of $u\in \mathbb{R}^p$ and $v\in \mathbb{R}^q$ such that \eqref{VS-3}, \eqref{VS-4}, and
\begin{equation}\label{KKTCondMan}
0\in \nabla F(\bar x,\bar y) + \nabla G(\bar x,\bar y)u + \nabla g(\bar x,\bar y)v + \lambda \nabla f(\bar x,\bar y) +
\lambda\left(\begin{array}{c}\partial (-\varphi)(\bar x)\\\{0\}\end{array}\right)
\end{equation}
hold for $(x,y)=(\bar x,\bar y)$. It is clear that \eqref{VS-2} for $(x,y)=(\bar x,\bar y)$  follows from the last $m$ components of this
inclusion.
Moreover, considering the first $n$ components of {\eqref{KKTCondMan}}, we get
\[
\nabla_1 F(\bar x,\bar y) + \nabla_1 G(\bar x,\bar y)u + \nabla_1 g(\bar x,\bar y)v + \lambda \nabla_1 f(\bar x,\bar y) \in \lambda\partial \varphi(\bar x)
\]
because of $\partial(-\varphi)(\bar x)= -\partial\varphi(\bar x)$.
Now, since the lower-level functions are fully convex and LMFCQ holds at $\bar x$, it follows from \cite[Theorem 2.1]{TaninoOgawa1984} that we can find some $z\in S(\bar x)$ and $w\in \Lambda(\bar x, z)$ such that \eqref{KS-1} holds for $(x,y)=(\bar x,\bar y)$. Furthermore, observe that the conditions \eqref{KS-2} and \eqref{KS-3} for $(x,y)=(\bar x,\bar y)$ result from the definition of
$w\in \Lambda(\bar x, z)$; cf. \eqref{Lambda(x,y)}.
\end{proof}
\begin{rem}
In the literature, for example, see \cite{DempeDuttaMordukhovichNewNece, DempeZemkohoGenMFCQ} and references therein,
an upper-level regularity condition is often used in combination with LMFCQ to derive \eqref{VS-3}, \eqref{VS-4}, and \eqref{KKTCondMan}.
Here, we employ UMFCQ instead as problem \eqref{P} involves  coupled upper-level constraints (i.e., depending on both the upper and lower-level variables), which is not the case in the aforementioned papers. Moreover, due to the convexity of the lower-level problem, the explicit use of inclusion $z\in S(\bar x)$ is avoided, given that its fulfillment follows from the conditions \eqref{KS-2} and \eqref{KS-3}.
\end{rem}

\begin{rem}\label{fallom}
If  $y=z$ in the optimality conditions in Theorem \ref{KN stationarity partial calm}, then we arrive at another well-known type of conditions consisting of \eqref{VS-2}--\eqref{KS-3} and
\begin{eqnarray}
 \nabla_1F(x,y) + \nabla_1 G(x,y)u  +  \nabla_1 g(x,y)\left(v - \lambda w\right)  +  \lambda \nabla_1 f(x,y) =0 \label{KS-1*}
\end{eqnarray}
with $z$  replaced by $y$ in \eqref{KS-2} and \eqref{KS-3}. Instead of the full convexity assumption, the inner-semicontinuity concept can also allow one to derive these conditions. To see this, note that if the LMFCQ and inner semicontinuity both hold at $(\bar x, z)$, then the Clarke subdifferential of $\varphi$ can be estimated as
\[
\partial \varphi(\bar x) \subseteq \left\{\nabla_1 \ell(\bar x, z, w)\left|\;\; w\in \Lambda(\bar x, z)\right.\right\};
\]
see \cite{DempeDuttaMordukhovichNewNece,DempeZemkohoGenMFCQ} for related details and references. Note that various other stationarity concepts for the bilevel programs based on the LLVF reformulation are possible; see the latter references for related details. However, it is important to note for any of these conditions, the convexity assumption will still be required for the lower-level problem for inclusion $y\in S(x)$ to be ignored in these conditions. 
\end{rem}
For conditions ensuring that  the assumptions in Theorem \ref{KN stationarity partial calm}, in particular partial calmness, we refer to \cite{DempeDuttaMordukhovichNewNece,DempeZemkohoGenMFCQ,DempeZemkohoBlpRefCqOptCond,YeZhuOptCondForBilevel1995, YeZhuZhuExactPenalization1997} and references therein.
Keeping the penalty parameter $\lambda>0$ fixed, the optimality conditions in Theorem \ref{KN stationarity partial calm} can be regarded as a mixed complementarity system and lead to an equivalent square system of nonsmooth equations which is dealt with a semismooth Newton method \cite{DeLuca1996,FischerASpecial1992,Kummer1988,QiJiangSemismooth1997,Qi1993convergence, QiSunANonsmoothVersion1993}. Dealing with other optimality conditions, like in \cite{DempeDuttaMordukhovichNewNece, DempeZemkohoGenMFCQ, YeZhuOptCondForBilevel1995}, in a similar way, one is led to more difficult
nonsmooth systems for which more sophisticated Newton-type methods \cite{FFH2013,FHI2016,FHI2018} might be helpful. 

\section{The algorithm}\label{The algorithm}

To apply the semismooth Newton method from \cite{DeLuca1996} to system
\eqref{KS-1}--\eqref{KS-3}, for some fixed $\lambda>0$, the latter system is reformulated by
means of the complementarity function
$\phi:\R^2\to\R$ \cite{FischerASpecial1992} given by
\[
 \phi(a,b):=\sqrt{a^2+b^2}-a-b.
\]
It can be easily checked that $\phi(a,b)=0$ if and only if $a\ge 0$, $b\ge 0$, $ab=0$ is valid.
Therefore, to rewrite the complementarity system \eqref{VS-3}--\eqref{KS-3}, we first define
functions $\phi^G:\R^n\times\R^m\times\R^p\to\R^{p}$ and $\phi^g:\R^n\times\R^m\times\R^q\to\R^q$ by
\[
 \phi^G(x,y,u):=\left(
 \begin{array}{c}
  \phi(-G_1(x,y),u_1)\\\vdots\\\phi(-G_p(x,y),u_p)
 \end{array}
 \right)\quad\mbox{and}\quad\phi^g(x,y,v):=\left(
 \begin{array}{c}
  \phi(-g_1(x,y),v_1)\\\vdots\\\phi(-g_q(x,y),v_q)
 \end{array}
 \right),
\]
respectively. Furthermore, we introduce the Lagrange-type functions $\mathcal L^\lambda :\R^n\times\R^m\times\R^p\times\R^q\to\R$ and $L^\lambda:\R^n\times\R^m\times\R^m\times\R^p\times\R^q\times\R^q\to\R$ be respectively defined by
\begin{eqnarray}
\mathcal{L}^{\lambda}(x,y,u,v) & := & F(x,y) + u^\top G(x,y) + v^\top g(x,y) + \lambda f(x,y), \label{upper-level lagrangian}\\
L^{\lambda}(x,y,z,u,v,w) & := & \mathcal{L}^\lambda(x,y,u,v) - \lambda\ell(x,z,w), \label{Upper Lagrangian}
\end{eqnarray}
where $\lambda>0$ is the fixed penalty parameter.
Based on these definitions, we now introduce the mapping $\Phi^\lambda:\R^N\to\R^N$ with $N:=n+2m+p+2q$ by
\begin{equation}\label{Phi-lambda}
\Phi^{\lambda}(\zeta):=\left[
\begin{array}{c}
  \nabla L^{\lambda}(\zeta)\\
 \phi^G(x,y,u)\\
 \phi^g(x,y,v)\\
 \phi^g(x,z,w)\\
\end{array}
\right]\quad\mbox{with}\quad\zeta:=(x, y, z, u, v, w),
\end{equation}
where $\nabla L^{\lambda}$ denotes the gradient of  $L^\lambda$ w.r.t. $(x,y,z)$.
Now, keeping \eqref{Upper Lagrangian} and \eqref{Phi-lambda} in mind, it can be easily seen that the optimality conditions \eqref{KS-1}--\eqref{KS-3} in Theorem \ref{KN stationarity partial calm} can equivalently be written as
\begin{equation}\label{Eq-Main}
    \Phi^{\lambda}(\zeta) = 0.
\end{equation}
Obviously, this is a square system of $N$ equations and $N$ variables. Moreover, the mapping $\Phi^\lambda$ is strongly semismooth at any solution of \eqref{Eq-Main}. In particular, we can apply the semismooth Newton method in \cite{DeLuca1996} with its favorable combination of global and local convergence properties. The latter means superlinear or quadratic convergence based on (strong) semismoothness of $\Phi^\lambda$ and a regularity property of the generalized Jacobians $\partial_B\Phi^\lambda$ or $\partial\Phi^\lambda$ at a solution $\zeta^*$ of \eqref{Eq-Main}. Sufficient conditions for $\partial\Phi^\lambda(\zeta^*)$
containing only nonsingular matrices will be developed in Section \ref{CD-regularity}. For global convergence, the complementarity function $\phi$ has the property that $\phi^2$ is differentiable with Lipschitz continuous derivative. Due to this, the merit function $\Psi^\lambda:\R^N\to\R$ with
\begin{equation}\label{Merit function}
    \Psi^{\lambda}(\zeta) := \frac{1}{2}\|\Phi^\lambda(\zeta)\|^2
\end{equation}
is continuously differentiable. In particular, this enables to overcome situations where
the Newton direction for \eqref{Eq-Main} does not exist or its descent is insufficient.
For other complementarity and merit functions as well as their properties we refer to \cite{FiJ2000, SunQi1999}, for example.

We now present the semismooth Newton algorithm from \cite{DeLuca1996} applied to equation
\eqref{Eq-Main} or, in other words, to the optimality conditions from Theorem \ref{KN stationarity partial calm}.

Some remarks on Algorithm \ref{algorithm 1} are in order. Compared to \cite{DeLuca1996}, we are
now dealing with a complementarity system instead of a pure complementarity problem. The penalization parameter $\lambda>0$ has to be chosen in Step 0 and is fixed throughout the algorithm. For the definition of
the generalized Jacobian $\partial_B\Phi^\lambda(\zeta)$, see Section \ref{prelim}. If all matrices in $\partial_B\Phi^\lambda(\zeta)$ are nonsingular, the function $\Phi^\lambda$ is called BD-regular at $\zeta$.
Following \cite{DeLuca1996}, global and local convergence properties of Algorithm \ref{algorithm 1} can be derived as follows.
\begin{algorithm}[H]
\caption{Semismooth Newton algorithm}
\label{algorithm 1}
\begin{algorithmic}
 \STATE\textbf{Step 0}: Choose $\lambda>0$, $\beta>0$, $\epsilon\ge 0$, $t>2$, $\rho\in(0,1)$, $\sigma\in \left(0, \, \frac{1}{2}\right)$, $\zeta^o:=(x^o, y^o, z^o,  u^o, v^o, w^o)$.\\
 \hspace*{3.45em}Set $k:=0$.
 \STATE\textbf{Step 1}: If $\|\Phi^{\lambda}(\zeta^k)\|\le\epsilon$, then stop.
 \STATE \textbf{Step 2}: Choose $W^k\in \partial_B \Phi^{\lambda}(\zeta^k)$ and compute a solution $d^k$ of the system
 \[
  W^k d =-\Phi^{\lambda}(\zeta^k).
 \]
 \hspace*{3.45em}If this equation has no solution or if the condition
 \[
 \nabla \Psi^{\lambda}(\zeta^k)^\top d^k \leq -\beta \|d^k\|^t
 \]
 \hspace*{3,45em}is violated, then set
 \[
 d^k:= -\nabla \Psi^{\lambda}(\zeta^k).
 \]
\STATE \textbf{Step 3}: Find the smallest integer $s_k\in\{0,1,2,3,\ldots\}$ such that
\[
 \Psi^{\lambda}(\zeta^k + \rho^{s_k}d^k)\le\Psi^{\lambda}(\zeta^k)+\sigma \rho^{s_k}\nabla\Psi^\lambda(\zeta^k)^\top d^k.
\]
\STATE \textbf{Step 4}: Set $\alpha_k:= \rho^{s_k}$, $\zeta^{k+1}:=\zeta^k + \alpha_k d^k$, $k:=k+1$, and go to Step 1.
\end{algorithmic}
\end{algorithm}
%
\begin{thm}\label{convergence result}
Assume that $\bar \zeta$ is an accumulation point of a sequence $\{\zeta^k\}$ generated by Algorithm \ref{algorithm 1} for some $\lambda > 0$. Then, $\bar \zeta$ is a stationary point of $\Psi^\lambda$, i.e., $\nabla \Psi^\lambda(\bar \zeta)=0$. If $\bar \zeta$ solves $\Phi^\lambda (\zeta)=0$ and $\Phi^\lambda$ is BD-regular at $\bar \zeta$, then $\{\zeta^k\}$ converges to  $\bar \zeta$ superlinearly and, if the functions $F,G,f$, and $g$ defining problem \eqref{P} have locally Lipschitz continuous second order derivates, the convergence is Q-quadratic.
\end{thm}
The proof of \cite[Theorem 11]{DeLuca1996} can be easily extended to the complementarity system
$\Phi^\lambda(\zeta)=0$. Moreover, it is known that the continuity of the second-order derivatives of the problem functions $F,G,f$, and $g$ (as assumed in Section \ref{Introduction}) suffices instead of the semismoothness of these derivatives (as used in \cite{DeLuca1996} for showing that eventually the unit stepsize $\alpha_k=1$ is accepted).

The assumption that $\Phi^\lambda$ is BD-regular at $\bar\zeta$ used in Theorem \ref{convergence result} can be replaced by the stronger CD-regularity of $\Phi^\lambda$ at $\bar\zeta$. The latter requires the non-singularity of all matrices in  $\partial \Phi^{\lambda}(\bar \zeta)$. In the next section, we focus on the derivation of sufficient conditions for CD-regularity. The latter also allows a nice connection to the Robinson condition that we are going to discuss in Section \ref{Robinson-type condition}.
Let us finally note that conditions ensuring that a stationary point of a merit function is a solution of the underlying equation were extensively studied for complementarity problems. We do not want to dive into this subject but note that if just one element of $\partial\Phi^\lambda(\bar\zeta)$ is nonsingular then
$\nabla\Psi^\lambda(\bar\zeta)=0$ implies $\Phi^\lambda(\bar\zeta)=0$, for example see \cite[Section 4]{FFK1998}.

\section{CD-regularity}\label{CD-regularity}

To provide sufficient conditions guarantying that CD-regularity holds for $\Phi^\lambda$ \eqref{Phi-lambda}, we first provide an upper estimate of the generalized Jacobian of $\Phi^\lambda$ in the sense of Clarke \eqref{Clarke Subdifferential}. Recall that the index sets $\eta^i$, $\nu^i$ and $\theta^i$ with $i=1, 2, 3$ are defined in \eqref{multiplier sets} and \eqref{nu2nu3}.
\begin{thm}\label{Jacobian Phi}Let the functions $F$, $G$, $f$, and $g$ be twice continuously differentiable at $\zeta:=(x,y,z,u,v,w)$.  If $\lambda>0$, then $\Phi^\lambda$ is semismooth at $\zeta$ and any matrix $W^\lambda\in \partial \Phi^\lambda (\zeta)$ can take the form
\begin{equation}\label{Wlambda}
W^\lambda =\left[
\begin{array}{cccccc}
A^{\lambda}_{11} & \left(A^{\lambda}_{21}\right)^{\top} & -\lambda A^\top_{31}   & B^\top_{11} & B^\top_{21} & -\lambda B^\top_{31} \\
A^{\lambda}_{21} & A^{\lambda}_{22} & O   & B^\top_{12} & B^\top_{22} & O \\
-\lambda A_{31} & O & -\lambda A_{33}   & O & O & -\lambda B^\top_{33} \\
\Lambda_1B_{11} & \Lambda_1B_{12} & O   & \Gamma_1 & O & O \\
\Lambda_2 B_{21} & \Lambda_2B_{22} & O   & O & \Gamma_2 & O \\
\Lambda_3B_{31} & O & \Lambda_3B_{33}   & O & O & \Gamma_3
\end{array}
\right]
\end{equation}
where the matrices $A_{ij}$ and $B_{ij}$ are respectively defined by
\begin{equation}\label{Blambda}
\begin{array}{l}
\begin{array}{l}
   A^{\lambda}_{11}:= \nabla^2_{1}\mathcal{L}^{\lambda}(\zeta)-\lambda \nabla^2_{1}\ell(\zeta), \;\;\, A^{\lambda}_{21}:= \nabla^2_{12}\mathcal{L}^{\lambda}(\zeta), \;\;\, A^{\lambda}_{22}:=\nabla^2_{2}\mathcal{L}^{\lambda}(\zeta), \\
   A_{31} :=  \nabla^2_{12}\ell(\zeta), \;\;\, A_{33}:= \nabla^2_{2}\ell(\zeta),
\end{array}\\
    \begin{array}{lll}
  B_{11} := \nabla_1 G(x,y), & B_{21} :=\nabla_1 g(x,y), & B_{31} :=\nabla_1 g(x,z),\\
   B_{12} :=\nabla_2 G(x,y), & B_{22} :=\nabla_2 g(x,y), & B_{33} :=\nabla_2 g(x,z),
\end{array}
\end{array}
\end{equation}
while  $\Lambda_i :=\mbox{diag} (a^i) $ and $\Gamma_i :=\mbox{diag}(b^i)$, $i=1, 2, 3$, are such that
 \begin{equation}\label{ab definition}
    (a^i_j,b^i_j)\left\{\begin{array}{ll}
                  =(0,-1) & \mbox{ if } \;j\in  \eta^i, \\
                  =(1,0) & \mbox{ if } \;j\in \nu^i, \\
                  \in \{(\alpha, \beta): \; (\alpha-1)^2 + (\beta+1)^2\leq 1\} & \mbox{ if }\; j\in \theta^i.
                \end{array}
\right.
 \end{equation}
\end{thm}
In the next result, we provide conditions ensuring that the function $\Phi^{\lambda}$ is CD-regular. To proceed, first note that, similarly to the UMFCQ \eqref{MFCQ-UL} and analogously to the LLICQ \eqref{LICQ}, the  {upper-level linear independence constraint qualification} (ULICQ) will be said to hold at  $(\bar x, \bar y)$ if the following family of vectors is linear independent:
\begin{equation}\label{LICQ-leader}
\left\{\nabla G_i(\bar x,\bar y):\; i\in I^1\right\} \cup  \left\{\nabla g_j(\bar x,\bar y):\; j\in I^2\right\}.
\end{equation}
Furthermore, let us introduce the cone of feasible directions for problem \eqref{Penalized-LLVF}
\begin{equation}\label{Ponana}
\begin{array}{llll}
Q(\bar x, \bar y, \bar z)  & := & \Big\{d\;\,\Big|&\nabla G_i(\bar x, \bar y)^\top d^{1,2}=0, \; i\in \nu^1,\\
          &    &                 &\nabla g_j(\bar x, \bar y)^\top d^{1,2}=0, \; j\in \nu^2,\;\;\;\nabla g_j(\bar x, \bar z)^\top d^{1,3}=0, \; j\in \nu^3\Big\},
\end{array}
\end{equation}
where $d:=\left(d^1_1, \ldots, d^1_n, \; d^2_1, \ldots, d^2_m,\; d^3_1, \ldots, d^3_m \right)^\top$, $d^{1, 2}:=\left(d^1_1, \ldots, d^1_n,\; d^2_1, \ldots, d^2_m\right)^\top$, and $d^{1, 3}$ defined similarly. Recall that for  $i=1, 2, 3$, $\nu^i$ is defined as in \eqref{multiplier sets}--\eqref{nu2nu3}.
By $\nabla^2\mathcal{L}^{\lambda}(\bar \zeta)$ and $\nabla^2\ell(\bar \zeta)$, we will denote the Hessian of the Lagrangian functions  $\mathcal{L}^{\lambda}$ and $\ell$ w.r.t. $(x,y)$ and $(x,z)$, respectively.

\begin{thm}\label{SOSSC-Theorem 1-1}{
Let $\bar \zeta:=(\bar x, \bar y, \bar z, \bar u, \bar v, \bar w)$ satisfy the conditions \eqref{KS-1}--\eqref{KS-3} for some $\lambda >0$. Suppose that  ULICQ \eqref{LICQ-leader} and LLICQ \eqref{LICQ} hold at $(\bar x, \bar y)$ and $(\bar x, \bar z)$, respectively. If  additionally, 
  \begin{equation}\label{SOSSC}
  \begin{array}{l}
 (d^{1,2})^{\top} \nabla^2\mathcal{L}^{\lambda}(\bar \zeta) d^{1,2} > \lambda (d^{1,3})^{\top} \nabla^2\ell(\bar \zeta) d^{1,3}
  \end{array}
  \end{equation}
for all $d\in Q(\bar x, \bar y, \bar z)\setminus \{0\}$ and LSCC \eqref{LSCC} is also satisfied at $(\bar x, \bar z, \bar w)$, then  $\Phi^{\lambda}$ is CD-regular at $\bar \zeta$.
}
\end{thm}
\begin{proof}
Let $W^\lambda$ be any element from $\partial \Phi^{\lambda}(\bar \zeta)$. Then, it can take the form described in Theorem \ref{Jacobian Phi}, cf. \eqref{Wlambda}--\eqref{ab definition}. Hence, it follows that for any $d:=(d^1, d^2, d^3, d^4, d^5, d^6)$ with $d^1\in \mathbb{R}^n$, $d^2\in \mathbb{R}^m$, $d^3\in \mathbb{R}^m$, $d^4\in \mathbb{R}^p$, $d^5\in \mathbb{R}^q$ and $d^6\in \mathbb{R}^q$ such that $W^\lambda d =0$, we have
\begin{eqnarray}
  \nabla^2_{1} \mathcal{L}^\lambda (\bar\zeta)d^1 - \lambda \nabla^2_{1} \ell (\bar\zeta)d^1 +\nabla^2_{21}\mathcal{L}^\lambda (\bar\zeta) d^2 -\lambda \nabla^2_{21}\ell(\bar\zeta) d^3 \qquad \nonumber\, \,\\
  + \;\nabla_1 G(\bar x, \bar y)^\top d^4 + \nabla_1 g(\bar x,\bar y)^\top d^5 -\lambda \nabla_1 g(\bar x, \bar z)^\top d^6=0,\label{p11}\\
  \nabla^2_{12} \mathcal{L}^\lambda (\bar \zeta)d^1 +\nabla^2_{2}\mathcal{L}^\lambda (\bar \zeta) d^2 + \nabla_2 G(\bar x, \bar y)^\top d^4 + \nabla_2 g(\bar x, \bar y)^\top d^5=0,\label{p12}\\
  -\;\lambda\nabla^2_{12}\ell (\bar \zeta)d^1 - \lambda \nabla^2_{2} \ell (\bar \zeta)d^3 -\lambda \nabla_2 g(\bar x, \bar z)^\top d^6=0,\label{p13}\\
  \forall j=1, \ldots, p, \; a^1_j \nabla G_j(\bar x, \bar y)^\top d^{1,2} + b^1_j d^4_j =0,\label{p14}\\
    \forall j=1, \ldots, q, \; a^2_j \nabla g_j(\bar x, \bar y)^\top d^{1,2} + b^2_j d^5_j =0,\label{p15}\\
     \forall j=1, \ldots, q, \; a^3_j \nabla g_j(\bar x, \bar z)^\top d^{1,3} + b^3_j d^6_j =0.\label{p16}
\end{eqnarray}
Recall that $p$ and $q$ represent the number of components of upper- (resp. lower-) constraint functions \eqref{P}--\eqref{lower-level problem}. For $i=1, 2, 3$, let $p_1:=p$, $p_2:=q$ (when $g$ applied to $(\bar x, \bar y)$), and $p_3:=q$ (when $g$ applied to $(\bar x, \bar z)$). Then define $P^i_1$ as the set of indices $j=1 \ldots, p_i$ such that  $a^i_j>0$ and $b^i_j<0$; $P^i_2$ as the set of indices $j=1 \ldots, p_i$ such that  $a^i_j=0$ and $b^i_j=-1$; and $P^i_3$ as the set of indices $j=1 \ldots, p_i$ such that  $a^i_j=1$ and $b^i_j=0$.
It follows from \eqref{p14}--\eqref{p16} that for $j\in P^1_2$, $j\in P^2_2$, and $j\in P^3_2$,
\begin{equation}\label{yes1}
    d^4_j=0, \; d^5_j=0,\, \mbox{ and }d^6_j=0,
\end{equation}
respectively. As for $j\in P^1_3$, $j\in P^2_3$, and $j\in P^3_3$, we respectively get
\begin{equation}\label{yes2}
    \nabla G_j (\bar x,\bar y)^\top d^{1,2}=0,\; \nabla g_j (\bar x,\bar y)^\top d^{1,2}=0, \, \mbox{ and }\, \nabla g_j (\bar x,\bar z)^\top d^{1,3}=0.
\end{equation}
Now observe that under the LSCC \eqref{LSCC},  $\theta^3=\emptyset$. Hence, from the corresponding counterpart of \eqref{ab definition}, it follows that $P^3_1:=\emptyset$. We can further check that  for $j\in P^1_1$ and $j\in P^2_1$ we respectively have
\begin{equation}\label{yes3}
    \nabla G_j (\bar x,\bar y)^\top d^{1,2}=c^1_j d^4_j \,\mbox{ and }\,  \nabla g_j (\bar x,\bar y)^\top d^{1,2}=c^2_j d^5_j,
\end{equation}
where $c^1_j:=-\frac{b^1_j}{a^1_j}$ and $c^2_j:=-\frac{b^2_j}{a^2_j}$, respectively.
By respectively multiplying \eqref{p11}, \eqref{p12}, and \eqref{p13} from the left by $(d^1)^\top$, $(d^2)^\top$, and $(d^3)^\top$, and adding the resulting sums together,
\begin{equation}\label{Quad-term}
    \begin{array}{l}
 (d^{1,2})^\top \nabla^2 \mathcal{L}^\lambda (\bar\zeta) d^{1,2} - \lambda (d^{1,3})^\top \nabla^2 \ell (\bar\zeta) d^{1,3} \\
\qquad \qquad \qquad +  \;\, (d^4)^\top\nabla G(\bar x,\bar y) d^{1,2}
  +
  (d^5)^\top\nabla g(\bar x,\bar y) d^{1,2}
  -\lambda
  (d^6)^\top\nabla g(\bar x,\bar z) d^{1,3}=0.
\end{array}
\end{equation}
Considering the strict complementarity slackness at $(\bar x, \bar z, \bar w)$ again, it follows that
\begin{equation}\label{SCS-applied}
 (d^6)^\top\nabla g(\bar x,\bar z) d^{1,3} =  \sum_{j\in P^3_{2}}d^6_j\nabla g_j(\bar x,\bar z)^{\top} d^{1,3} + \sum_{j\in P^3_{3}}d^6_j\nabla g_j(\bar x,\bar z)^{\top} d^{1,3}=0
\end{equation}
 given that $d^6_j=0$ for $j\in P^3_{2}$ and $\nabla g_j(\bar x,\bar z)^{\top} d^{1,3}=0$ for $j\in P^3_{3}$.
 Inserting \eqref{SCS-applied} into \eqref{Quad-term} while taking into account \eqref{yes1}--\eqref{yes3}, we get
 \begin{equation*}\label{SOSSC-needed}
(d^{1,2})^\top \nabla^2 \mathcal{L}^\lambda (\bar \zeta) d^{1,2} - \lambda (d^{1,3})^\top \nabla^2 \ell (\bar\zeta) d^{1,3}
  + \sum_{j\in P^1_1}c^1_j (d^4_j)^2 + \sum_{j\in P^2_1}c^2_j (d^5_j)^2 =0.
 \end{equation*}
Since by definition, $c^1_j >0$ for $j\in P^1_1$ and $c^2_j >0$ for $j\in P^2_1$, it follows from condition \eqref{SOSSC} that $d^{1}=0$, $d^{2}=0$, $d^{3}=0$, $d^4_j=0$ for $j\in P^1_1$ and $d^5_j=0$ for $j\in P^2_1$, while taking into account \eqref{yes2} and the fact that $\nu^i \subseteq P^i_3$ for $i=1, 2, 3$.
Inserting these values in \eqref{p11}--\eqref{p13} and considering \eqref{yes1},
 \begin{eqnarray}
 \sum_{j\in P^1_3}d^4_j\nabla_1 G_j(\bar x,\bar y) + \sum_{j\in P^2_3}d^5_j\nabla_1 g_j(\bar x,\bar y) + \sum_{j\in P^3_3}(-\lambda d^6_j)\nabla_1 g_j(\bar x,\bar z)=0,\label{p111}\\
   \sum_{j\in P^1_3}d^4_j\nabla_2 G_j(\bar x,\bar y) + \sum_{j\in P^2_3}d^5_j \nabla_2 g_j(\bar x,\bar y)=0,\label{p122}\\
  \sum_{j\in P^3_3}d^6_j\nabla_2 g_j(\bar x,\bar z)=0.\label{p133}
\end{eqnarray}
Since the LLICQ \eqref{LICQ} is satisfied at $(\bar x, \bar z)$ and $P^3_3 \subseteq I^3$ holds, it follows from \eqref{p133} that  $d^6_j=0$ for $j\in P^3_{3}$. Inserting these values in \eqref{p111} and combining the resulting equation with \eqref{p122},
\begin{equation}\label{p122-1}
    \sum_{j\in P^1_3}d^4_j\nabla G_j(\bar x,\bar y) + \sum_{j\in P^2_3}d^5_j \nabla g_j(\bar x,\bar y)=0.
\end{equation}
Considering the fulfillment of the ULICQ \eqref{LICQ-leader}  at $(\bar x, \bar y)$, and taking into account that $P^i_3 \subseteq I^i$ for $i=1, 2$, we can deduce from \eqref{p122-1} that $d^4_j=0$ for $j\in P^1_3$ and $d^5_j=0$ for $j\in P^2_3$. This concludes the proof as we have shown that all the components of $d$ are zero.
\end{proof}
Note that the strict complementarity condition imposed here is restricted to the lower-level problem and does not necessarily imply the local differentiability of the lower-level optimal solution map as known in earlier results on the Newton method, see, e.g., \cite{KleinmichelSchonefeld1988,KojimaHirabayashi1984} or in the literature on nonlinear parametric optimization, see, e.g., \cite{FiaccoBook1983}.   As we also have the LLICQ, the differentiability of lower-level optimal solution is usually guarantied when a strong second order sufficient condition (SOSSC)-type condition restricted to the lower-level problem is satisfied as well. The LSCC here only allows us to deal with the minus sign appearing on the lower-value function in problem \eqref{LLVF} and is responsible for many of the stationarity concepts for the problem; cf. \cite{DempeDuttaMordukhovichNewNece, DempeZemkohoGenMFCQ, YeZhuOptCondForBilevel1995}.
Next we discuss two possible scenarios to avoid imposing the LSCC. In the first case, we assume that the lower-level feasible set is unperturbed.
\begin{thm}\label{SOSSC-Theorem 2}{Let $g$ in problem \eqref{P} be independent of the upper-level variable $x$
and suppose that the point $\bar \zeta:=(\bar x, \bar y, \bar z, \bar u, \bar v, \bar w)$ satisfies conditions \eqref{KS-1}--\eqref{KS-3} for some $\lambda >0$. Furthermore, assume that the family $\left\{\nabla_x G_j(\bar x,\bar y):\; j\in I^1\right\}$ is linearly independent and the LLICQ holds at $\bar y$ and $\bar z$.
 If additionally,
  \begin{equation}\label{SOSSC-1}
  \begin{array}{l}
 (d^{1,2})^{\top} \nabla^2\mathcal{L}^{\lambda}(\bar \zeta) d^{1,2} > \lambda (d^{1,3})^{\top} \nabla^2\ell^*(\bar \zeta) d^{1,3}
  \end{array}
  \end{equation}
 for all $d \in Q(\bar \zeta)\setminus \{0\}$, where
\begin{equation}\label{barell}
\nabla^2 \ell^*(\bar \zeta):= \left[\begin{array}{rr}
                                                                              \nabla^2_{1}\ell(\bar \zeta) & \nabla^2_{21}\ell(\bar \zeta) \\
                                                                              -\nabla^2_{12}\ell(\bar \zeta) & -\nabla^2_{2}\ell(\bar \zeta)
                                                                            \end{array}
                                                                                 \right],
\end{equation}
then the function $\Phi^{\lambda}$ is CD-regular at the point $\bar \zeta:=(\bar x, \bar y, \bar z, \bar u, \bar v, \bar w)$.
}
\end{thm}
\begin{proof} Considering the counterpart of \eqref{p11}--\eqref{p16} when $g$ is independent of $x$ and proceeding as in the proof of the previous theorem, we have  \eqref{yes1}, \eqref{yes2} and
\begin{equation}\label{yes3thm34}
    \nabla G_j (\bar x,\bar y) d^{1,2}=c^1_j d^4_j, \;  \nabla g_j (\bar y) d^{2}=c^2_j d^5_j, \; \mbox{ and }\,\nabla g_j (\bar z) d^{3}=c^3_j d^6_j
\end{equation}
for $j\in P^1_1$, $j\in P^2_1$, and $j\in P^3_1$, respectively. Here, $c^1_j$ for  $j\in P^1_1$ and $c^2_j$ for  $j\in P^2_1$ are defined as in \eqref{yes2} while  $c^3_j:=-\frac{b^3_j}{a^3_j}$ for $j\in P^3_1$, cf. \eqref{ab definition}. Next, replacing the counterpart of \eqref{p13} with
$$
\nabla^2_{12} \ell(\bar\zeta)d^1 + \nabla^2_{2} \ell (\bar \zeta)d^3 + \nabla g(\bar z)^\top d^6=0
$$
and  multiplying this equality, \eqref{p12}, and \eqref{p11} from the left by $(d^3)^\top$, $(d^2)^\top$, and $(d^1)^\top$, respectively, and adding the resulting sums together, we obtain
\begin{equation*}
(d^{1,2})^\top \nabla^2 \mathcal{L}^\lambda (\bar\zeta) d^{1,2} - \lambda (d^{1,3})^\top \nabla^2 \ell^* (\bar\zeta) d^{1,3}
  + \sum_{j\in P^1_1}c^1_j (d^4_j)^2 + \sum_{j\in P^2_1}c^2_j (d^5_j)^2 + \sum_{j\in P^3_1}c^3_j (d^6_j)^2=0,
\end{equation*}
while taking into account \eqref{yes3thm34} and the counterparts of \eqref{yes1} and \eqref{yes2}. Hence, it follows from assumption \eqref{SOSSC-1} that $d^{1}=0$, $d^{2}=0$, $d^{3}=0$, $d^4_j=0$ for $j\in P^1_1$, $d^5_j=0$ for $j\in P^2_1$, and $d^6_j=0$ for $j\in P^3_1$.
The rest of the proof then follows as that of Theorem \ref{SOSSC-Theorem 1-1}.
\end{proof}
For the next result, the LSCC is also not needed, \af{and} $g$ does not necessarily have to be independent of the upper-level variable.
\begin{thm}\label{SOSSC-Theorem 3}
Let
$\bar \zeta:=(\bar x, \bar y, \bar z, \bar u, \bar v, \bar w)$ satisfy the optimality conditions \eqref{KS-1}--\eqref{KS-3} for some $\lambda >0$. Suppose that the ULICQ and LLICQ hold at $(\bar x, \bar y)$ and $(\bar x, \bar z)$, respectively. Then $\Phi^{\lambda}$ is CD-regular at $\bar \zeta:=(\bar x, \bar y, \bar z, \bar u, \bar v, \bar w)$ provided we also have
  \begin{equation}\label{SOSSC-2}
 (d^{1,2})^{\top} \nabla^2 \mathcal{L}^\lambda (\bar \zeta) d^{1,2} > \lambda \left\{(d^{1,3})^{\top} \nabla^2 \ell (\bar \zeta) d^{1,3} + \sum_{j\in P^3_1}c^3_j (e_j)^2\right\}
  \end{equation}
 for all $(d, e)\in \left[Q(\zeta)\times \mathbb{R}^{|P^3_1|}\right]\setminus \{0\}$, with $c^3_j:=-\frac{b^3_j}{a^3_j}$ for $j\in P^3_1$; cf.  \eqref{ab definition}.
\end{thm}
\begin{proof}
Also proceeding as in the proof of Theorem \ref{SOSSC-Theorem 1-1} while replacing \eqref{yes3} with
\begin{equation}\label{yes33}
    \nabla G_j (\bar x,\bar y) d^{1,2}=c^1_j d^4_j, \;  \nabla g_j (\bar x,\bar y) d^{1,2}=c^2_j d^5_j, \,\mbox{ and }\, \nabla g_j (\bar x,\bar z) d^{1,2}=c^3_j d^6_j
\end{equation}
for $j\in P^1_1$, $j\in P^2_1$,  and $j\in P^3_1$, respectively, we get equality
 $$
 \begin{array}{l}
   (d^{1,2})^\top \nabla^2 \mathcal{L}^\lambda (\zeta) d^{1,2} - \lambda \left\{(d^{1,3})^\top \nabla^2 \ubar{\ell} (\zeta) d^{1,3} + \sum_{j\in P^3_1}c^3_j (d^6_j)^2\right\} + \sum_{j\in P^1_1}c^1_j (d^4_j)^2 + \sum_{j\in P^2_1}c^2_j (d^5_j)^2 =0
 \end{array}
 $$
by inserting \eqref{yes1}--\eqref{yes2} and \eqref{yes33} in the counterpart of \eqref{Quad-term}, as $\theta^3$ is not necessarily empty. Hence, under assumption \eqref{SOSSC-2}, we get $d^{1}=0$, $d^{2}=0$, $d^{3}=0$, $d^4_j=0$ for $j\in P^1_1$, $d^5_j=0$ for $j\in P^2_1$ and $d^6_j=0$ for $j\in P^3_1$. Similarly, the rest of the proof then follows as for Theorem \ref{SOSSC-Theorem 1-1}.
\end{proof}
Considering the structure of the generalized second order subdifferential of $\varphi$ \eqref{varphi} (see \cite{zemkoho2017estimates}), condition \eqref{SOSSC-2}  can be seen as the most natural extension to our problem \eqref{Penalized-LLVF} of the strong second order sufficient condition used for example in \cite{FischerASpecial1992,QiJiangSemismooth1997}. To see this, note that condition \eqref{SOSSC-2} can be replaced by the following condition, for all $(d^{1,2,3}, e)$ in  $ \left[Q(\zeta)\times \mathbb{R}^q\right]\setminus \{0\}$:
  \begin{equation*}\label{SOSSC-1-2-3-4}
  \begin{array}{l}
 (d^{1,2})^{\top} \nabla^2\mathcal{L}^\lambda(\bar \zeta) d^{1,2} > \lambda \left\{(d^{1,3})^{\top} \nabla^2\ell(\bar \zeta) d^{1,3} + e^\top\nabla g(\bar x,\bar z) d^{1,3}\right\}.
  \end{array}
  \end{equation*}
\begin{example}\label{example78}
Consider the bilevel optimization problem
\begin{equation}\label{example-pb}
  \begin{array}{l}
\underset{x,y}\min~x^2 + y^2_1 + y^2_2\\
\mbox{s.t. }\;y\in S(x):=\arg\underset{y}\min~\left\{\|y-(x, -1)^\top\|^2:\; y_1-y_2\leq 0, \; -y_1 - y_2\leq 0\right\},
\end{array}
\end{equation}
where the lower-level problem is taken from \cite[Chapter 1]{FiaccoBook1983}.
The LLVF \eqref{varphi} can be obtained as
$$
\varphi(x) = \left\{\begin{array}{ll}
                     \frac{1}{2} (1-x)^2 & \mbox{if } x < -1, \\
                     1+x^2   & \mbox{if }  -1 \leq x \leq 1,\\
                     \frac{1}{2} (1+x)^2 & \mbox{if } x > 1.
                   \end{array}
\right.
$$
The optimal solution of problem \eqref{example-pb} is $(\bar x, \bar y)$ with $\bar x=0$ and $\bar y=(0,0)$.
Considering the expression of $\varphi$ above, one can easily check that $(0, 0, 0, 0)$ satisfies the conditions
 $$
 y_1-y_2\leq 0, \;\; -y_1 - y_2\leq 0,\;\; \|y-(x, -1)^\top\|^2 - \varphi(x) + \varsigma=0.
 $$
Thus, for problem \eqref{example-pb}, condition \eqref{partial calmness} holds with $(\bar x, \bar y)=(0, 0, 0)$ and $U=\mathbb{R}^4$.
Next, note that the LMFCQ \eqref{MFCQ-follower} holds at any lower-level feasible point.  The lower-level optimal solution mapping $S$ is single valued and continuous in this case; hence inner-semicontinuous \cite{DempeDuttaMordukhovichNewNece}. Hence,  $(\bar x, \bar y)$ satisfies \eqref{KS-1}--\eqref{KS-3} and  subsequently, the corresponding calculations show that the point $\bar \zeta:=(\bar x, \bar y, \bar z, \bar v, \bar w)$, where
$\bar x= 0$, $\bar y=(0, 0)$, $\bar z:=(0, 0)$, $\bar v:=(\lambda, \lambda)$, and $\bar w=(1, 1)$ with $\lambda >0$,
solves  \eqref{Phi-lambda}.
Furthermore,  $Q(\bar x, \bar y, \bar z) = \mathbb{R}\times \{(0, 0, 0, 0)\}$ and for all $d^{1,2,3}\in Q(\bar x, \bar y, \bar z)\setminus \{0\}$,
$$
(d^{1,2})^{\top} \nabla^2\mathcal{L}^{\lambda}(\bar \zeta) d^{1,2} - \lambda (d^{1,3})^{\top} \nabla^2\ell(\bar \zeta) d^{1,3}= 2(d^1)^2>0.
$$
Hence, for problem \eqref{example-pb}, $\Phi^\lambda$ \eqref{Phi-lambda} is CD-regular at $\bar \zeta:=(\bar x, \bar y, \bar z, \bar v, \bar w)$, for any value of  $\lambda >0$.
\end{example}

\section{Robinson-type condition}\label{Robinson-type condition}
For a standard nonlinear optimization problem with twice continuously differentiable functions, the  {Robinson condition} \cite{Robinson1982} is said to hold at one of its KKT points if LICQ and a strong second order sufficient condition (SSOSC) are satisfied.
It was shown in \cite{QiJiangSemismooth1997} that if a standard nonlinear optimization is SC$^1$ and satisfies the Robinson condition, then the CD-regularity condition holds for the corresponding counterpart of function $\Phi^\lambda$ \eqref{Phi-lambda}. Considering the structure of the results from the previous section, it can be argued that the combination of assumptions in Theorems \ref{SOSSC-Theorem 1-1}, \ref{SOSSC-Theorem 2} or \ref{SOSSC-Theorem 3} corresponds to an extension of Robinson's condition to the context of bilevel optimization. However, another implication of Robinson's condition, i.e., precisely of the SSOSC, is that it ensures that a given point is a strict local optimal solution for the corresponding nonlinear programming problem.

The aim of this section is to enhance the second order assumption in the previous section so that it can guaranty that points computed by our algorithm are strict local optimal points. To proceed, we introduce the following cone of feasible directions
\begin{equation}\label{C(x,y)}
\begin{array}{ll}
         C^\lambda(\bar x, \bar y) :=   \big\{d\,|& \nabla G_i(\bar x, \bar y)^\top d = 0\, \mbox{ for }\, i\in \nu^1,\; \nabla G_i(\bar x, \bar y)^\top d \leq 0\, \mbox{ for }\, i\in \theta^1,\\
                                          & \nabla g_j(\bar x, \bar y)^\top d = 0\, \mbox{ for }\, j\in \nu^2,\; \nabla g_j(\bar x, \bar y)^\top d \leq 0\, \mbox{ for }\, j\in \theta^2,\\
                                                 &  \nabla F(\bar x, \bar y)^\top d + \lambda f(\bar x, \bar y)^\top d - \lambda \nabla_1 \ell (\bar x, z, w)^\top d^1 \leq 0\, \mbox{ for }\, z\in S(\bar x)\big\},
                          \end{array}
\end{equation}
where $\{w\} := \{w(z)\}=\Lambda(\bar x, z)$ for a fixed  $z\in S(\bar x)$, as the LLICQ \eqref{LICQ} will be assumed to hold at $(\bar x, z)$ for all $z\in S(\bar x)$. Also note that as in \eqref{Ponana}, $d$ can be written as $d:=((d^1)^\top, (d^2)^\top)^\top$.
Furthermore, we will use the following modified version of the upper-level Lagrangian function \eqref{upper-level lagrangian}
$$
 \bar{\mathcal{L}}^{\lambda}_{\kappa}(x,y,u,v):= \kappa \left(F(x,y) + \lambda f(x,y)\right) + \sum_{i\in I^1(d)}u_iG_i(x,y) +   \sum_{j\in I^2(d)}v_jg_j(x,y),
$$
 where  the set $I^1(d)$ (resp. $I^2(d)$) represents the set of indices $i\in I^1$ (resp. $j\in I^2$) such that we have $\nabla G_i(\bar x, \bar y)^\top d=0$ (resp. $\nabla g_j(\bar x, \bar y)^\top d=0$). Recall that $I^1$ and $I^2$ are given in \eqref{I1u} and \eqref{I1l}. In the next result, we first provide slightly general  SSOSC-type condition for problem \eqref{Penalized-LLVF}.
\begin{thm}\label{SuffNon-NEW}
Let the point $\bar\zeta:=(\bar x, \bar y, \bar z, \bar u, \bar v, \bar w)$ satisfy the conditions \eqref{KS-1}--\eqref{KS-3} for some $\lambda>0$. Suppose that the lower-level problem is convex at $\bar x$ (i.e., $f(\bar x, .)$ and $g_i(\bar x, .)$, $i=1, \ldots, q$, are convex) and the  assumptions in Theorem \ref{Theorem-phi''} hold for all $d\in C^\lambda(\bar x, \bar y)$. Then, $(\bar x, \bar y)$ is a strict local optimal solution of problem \eqref{Penalized-LLVF} provided that, for all $d\in C^\lambda(\bar x, \bar y)\setminus \{0\}$, the condition
\begin{equation}\label{SOSCB}
  \begin{array}{l}
d^\top \nabla^2\bar{\mathcal{L}}^{\lambda}_{\kappa_\circ}(\bar x, \bar y, u, v) d > \lambda\sum^{k}_{t=1}\kappa_t\xi_{d^1}(\bar x,  z^t)
  \end{array}
  \end{equation}
  is satisfied for some vectors $u$, $v$, $z^t$, and $\kappa_t$, with $z^t\in S_1(\bar x; d^1)$ and $\Lambda(\bar x, z^t)=\{w^t\}$, for $t=1, \ldots, k$, where $k$ is some natural number and $\kappa_{\circ}:= \sum^{k}_{i=1}\kappa_t$, such that we have
 \begin{eqnarray}
 \nabla_1\bar{\mathcal{L}}^{\lambda}_{\kappa_\circ}(\bar x, \bar y, u, v) - \lambda  \sum^{k}_{t=1}\kappa_t \nabla_1 \ell(\bar x, z^t, w^t)=0, \;\;
 \nabla_2\bar{\mathcal{L}}^{\lambda}_{\kappa_\circ}(\bar x, \bar y, u, v)=0,\label{onam1}\\
             \kappa_\circ + \sum_{i\in I^1(d)}u_i + \sum_{j\in I^2(d)}v_j =1,\;\;
\kappa_\circ \geq 0, \;\; u_i\geq 0\,\mbox{ for }\, i\in I^1(d),\;\; v_j\geq 0 \,\mbox{ for }\, j\in I^2(d).\label{onam3}
 \end{eqnarray}
\end{thm}
\begin{proof}
First, consider the optimization problem in \eqref{Penalized-LLVF} for the parameter $\lambda>$ for which $\bar\zeta$ satisfies the conditions \eqref{KS-1}--\eqref{KS-3}. This problem can obviously be rewritten as
$$
\begin{array}{l}
  \min~F(x,y) +\lambda\left(f(x,y) - \varphi(x)\right) \;\, \mbox{ s.t. } \;\, \psi(x,y)\leq 0,
\end{array}
$$
where the function $\psi$ is defined by $\psi(x,y):=\max~\left\{G_1(x,y), \ldots, G_p(x,y),\; g_1(x,y), \ldots, g_q(x,y)\right\}$. Next, consider the unconstrained optimization problem
\begin{equation}\label{NEWfunc}
\begin{array}{l}
\min~\phi_\lambda(x,y):= \max~\left\{\psi_\lambda(x,y), \;\, \psi(x,y)\right\}
\end{array}
\end{equation}
with $\psi_\lambda(x,y):= F(x,y)- F(\bar x, \bar y) +\lambda\left(f(x,y)-f(\bar x, \bar y)\right) - \lambda \left(\varphi(x)-\varphi(\bar x)\right)$. Based on \cite[Chapter 3]{BonnansShapiroBook2000}, it suffices to show that the function $\phi_\lambda$ satisfies the following three conditions:
\begin{enumerate}
  \item[${[a]}$] $\phi_\lambda$ is directionally differentiable and we have
  \begin{equation}\label{DirectGreater}
    \phi'_\lambda (\bar x, \bar y; d) \geq 0 \;\; \mbox{ for all }\;\; d\in \mathbb{R}^{n+m};
  \end{equation}
  \item[${[b]}$] $\phi_\lambda$ is twice directionally differentiable \eqref{2ndOrderDirectional} and  fulfills the condition
\begin{equation}\label{molio}
\underset{e\in \mathbb{R}^{n+m}}\inf~\phi_\lambda''(\bar x, \bar y; d, e) > 0 \;\mbox{ for all }\; d\neq 0 \;\mbox{ s.t. }\; \phi'_\lambda (\bar x, \bar y; d)=0;
\end{equation}
  \item[${[c]}$] $\phi_\lambda$ satisfies the second order epiregularity condition, i.e., for any $d\in \mathbb{R}^{n+m}$, $t\geq 0$, and any function (path) $e$ from $\mathbb{R}_+$ to $\mathbb{R}^n \times \mathbb{R}^m$ such that $te(t) \rightarrow 0$ as $t \downarrow 0$,
\begin{equation}\label{2ndEpi}
\phi_\lambda \left((\bar x, \bar y) + td + \frac{1}{2}t^2 e(t)\right) \geq \phi_\lambda (\bar x, \bar y) + t \phi'_\lambda (\bar x, \bar y; d) + \frac{1}{2} t^2 \phi''_\lambda\left(\bar x, \bar y; d, e(t)\right) + \circ(t^2).
\end{equation}
\end{enumerate}

To prove condition ${[a]}$, first note that $\psi$ is directionally differentiable, as the upper- and lower-level constraint functions are continuously differentiable. In fact,
\begin{equation}\label{psi'}
   \psi'(\bar x, \bar y; d) = \max~\left\{\nabla G_i(\bar x, \bar y)^\top d \mbox{ for } i\in I^1,\; \nabla g_j(\bar x, \bar y)^\top d \mbox{ for } j\in I^2 \right\}
\end{equation}
 for any $d\in \mathbb{R}^{n+m}$. As for $\psi_\lambda$, recalling that  $\{w\} := \{w(z)\}=\Lambda(\bar x, z)$ for  $z\in S(\bar x)$, thanks to the fulfillment of the LLICQ \eqref{LICQ} at $(\bar x, z)$ for all $z\in S(\bar x)$, it follows from Theorem \ref{Theorem-phi'}  that
\begin{equation}\label{F'-lambda}
   \psi'_\lambda(\bar x, \bar y; d) = \max~\left\{\nabla (F + \lambda f)(\bar x, \bar y)^\top d - \lambda\nabla_x\ell(\bar x, z, w)^\top d^1,\;\, z\in S(\bar x)\right\}
\end{equation}
 for any $d\in \mathbb{R}^{n+m}$. Now, considering the function $\phi_\lambda$ \eqref{NEWfunc}, it holds that for any $d\in \mathbb{R}^{n+m}$,
\begin{equation}\label{phi'lambda}
\phi'_\lambda(\bar x, \bar y; d)=\max~\left\{\psi'_\lambda(\bar x, \bar y; d),\; \; \psi'(\bar x, \bar y; d)\right\},
\end{equation}
as $\phi_\lambda(\bar x, \bar y)= \psi_\lambda(\bar x, \bar y)= 0$ and $\psi(\bar x, \bar y)=0$ if $I^1 \cup I^2 \neq \emptyset$; $\psi'_\lambda(\bar x, \bar y; d)$ and $\psi'(\bar x, \bar y; d)$ are given in \eqref{psi'} and \eqref{F'-lambda}, respectively. Next, observe that as $\phi'_\lambda (\bar x, \bar y; 0) =0$, condition \eqref{DirectGreater} is equivalent to
\begin{equation}\label{FirstOrderOptPsi}
    0\in \partial_d \phi'_\lambda(\bar x, \bar y; 0),
\end{equation}
provided that $\phi'_\lambda (\bar x, \bar y; .)$ is a convex function. This is indeed the case, as $\psi' (\bar x, \bar y; .)$ and $\psi'_\lambda (\bar x, \bar y; .)$ are both convex functions. Recall that in \eqref{FirstOrderOptPsi}, $\partial_d$ represents the subdifferential (in the sense of convex analysis) w.r.t. $d$. It therefore remains to show that we can find an element from $\partial_d \phi'_\lambda(\bar x, \bar y; 0)$ which is zero. To proceed, first recall that $\bar\zeta:=(\bar x, \bar y, \bar z, \bar u, \bar v, \bar w)$
fulfills \eqref{KS-1}--\eqref{KS-3} and let
$$
\varrho := 1 + \sum_{i\in I^1} \bar{u}_i + \sum_{j\in I^2} \bar{v}_j.
$$
We have $\varrho >0$ and subsequently, it holds that $\kappa_\circ + \sum_{i\in I^1} \bar{u}'_i + \sum_{j\in I^2} \bar{v}'_j=1$ and
$$
\begin{array}{l}
\Theta:=\kappa_\circ \left(\nabla F(x,y) +\lambda \nabla f(\bar x, \bar y) - \lambda \left[\begin{array}{c}
                                           \nabla_x \ell(\bar x, \bar z, \bar w)\\
                                           0
                                         \end{array}\right]\right) + \sum_{i\in I^1} \bar{u}'_i\nabla G_i(\bar x, \bar y) + \sum_{j\in I^2} \bar{v}'_j\nabla g_j(\bar x, \bar y) =0,
\end{array}
$$
with $\kappa_\circ := \frac{1}{\varrho}$, $\bar{u}'_i := \frac{1}{\varrho}\bar{u}_i$ for $i\in I^1$, $\bar{v}'_j := \frac{1}{\varrho} \bar{v}'_j$ for $j\in I^2$,  and $\bar w \in \Lambda(\bar x, \bar z)$. By the convexity of the lower-level problem at $\bar x$ and the fulfilment of the LLICQ at $(\bar x, \bar z)$,  it follows that inclusion $\bar w \in \Lambda(\bar x, \bar z)$ is equivalent to $\bar z\in S(\bar x)$. Furthermore, that we can easily show that \eqref{phi'lambda} can be rewritten as
\begin{equation}\label{melan}
\begin{array}{lll}
  \phi'_\lambda(\bar x, \bar y; d) & = &\max~\left\{\nabla G_i(\bar x, \bar y)^\top d \mbox{ for } i\in I^1,\right.\\
                                     && \quad \qquad \nabla g_i(\bar x, \bar y)^\top d \mbox{ for } i\in I^2, \\
                                    & & \quad \qquad \left.\nabla (F + \lambda f)(\bar x, \bar y)^\top d - \lambda\nabla_x\ell(\bar x, z, w)^\top d^1,\;\, z\in S(\bar x)\right\}.
\end{array}
\end{equation}
Hence, as $\phi'_\lambda(\bar x, \bar y; 0)=0$ and all the items in the max operator, regarded as functions of $d$, are convex and zero for $d=0$, it holds that
$$
\begin{array}{lll}
  \partial_d\phi'_\lambda(\bar x, \bar y; 0) & = &\mbox{conv}~\left\{\nabla G_i(\bar x, \bar y) \mbox{ for } i\in I^1,\right.\\
                                     && \quad \qquad \nabla g_j(\bar x, \bar y) \mbox{ for } j\in I^2, \\
                                    & & \quad \qquad \left. \nabla (F + \lambda f)(\bar x, \bar y) - \lambda\left[\begin{array}{c}
                                                                                                                   \nabla_1\ell(\bar x, z, w)\\ 0
                                                                                                                 \end{array}\right],\;\, z\in S(\bar x)\right\},
\end{array}
$$
given that $S(\bar x)$ is a compact set under the uniform compactness assumption made on the mapping $K$ \eqref{K(x)} in Theorem \ref{Theorem-phi'}.  It clearly  follows that $\Theta\in \partial_d \phi'_\lambda(\bar x, \bar y; 0)$ and $\Theta=0$.

To prove condition ${[b]}$, note that $\psi$ is second order directionally differentiable and
  \begin{equation}\label{Psio2ndDirect}
\begin{array}{l}
  \psi''(\bar x, \bar y; d, e)=   \max~\big\{G''_i(\bar x, \bar y; d, e) \mbox{ for } i\in I^1(d), \; g''_j(\bar x, \bar y; d,e) \mbox{ for } j\in I^2(d)\big\}
\end{array}
\end{equation}
for all $d, e\in \mathbb{R}^{n+m}$. Furthermore, for all $d, e\in \mathbb{R}^{n+m}$, it holds that
\begin{align}
\psi''_\lambda(\bar x, \bar y; d, e) \quad \overset{(1)}{=} & \quad \underset{t\downarrow 0}\lim \frac{1}{ t^2/2}\left\{\left[(F+\lambda f)\left((\bar x, \bar y) + td + \frac{1}{2}t^2e\right) - (F+\lambda f)(\bar x, \bar y)\right.\right.\nonumber \\
  &\qquad - \left.\left. t(F+\lambda f)'(\bar x, \bar y; d)\right] \right. - \lambda \left.\left[\varphi\left(\bar x + td^1 + \frac{1}{2}t^2e^1\right) - \varphi(\bar x) -    t \varphi'(\bar x; d^1)\right]\right\}\nonumber\\
   \quad\overset{(2)}{=} & \quad (F+\lambda f)''(\bar x, \bar y; d, e) - \lambda \varphi''(\bar x; d^1)\nonumber\\
    \quad\overset{(3)}{=}& \quad F''(\bar x, \bar y; d, e) + \lambda f''(\bar x, \bar y; d, e) - \lambda \,\underset{z\in S_1(\bar x;\, d^1)}\inf \Big\{\nabla_1\ell(\bar x, z, w)e^1 + \xi_{d^1} (\bar x, z)\Big\}\nonumber\\
  \quad\overset{(4)}{=} & \quad \underset{z\in S_1(\bar x;\, d^1)}\sup \Big\{F''(\bar x, \bar y; d, e) + \lambda f''(\bar x, \bar y; d, e) - \lambda \nabla_1\ell(\bar x, z, w)e^1 - \lambda\xi_{d^1} (\bar x, z)\Big\}\nonumber\\
   \quad\overset{(5)}{=} &\quad  \underset{z\in S_1(\bar x; d^1)}\max\Big\{F''(\bar x, \bar y; d, e) + \lambda f''(\bar x, \bar y; d, e) - \lambda \nabla_1\ell(\bar x, z, w)e^1 - \lambda\xi_{d^1} (\bar x, z)\Big\} \label{phi2ndDirectDolar}
\end{align}
with $d:=d^{1,2}$, $e:=e^{1,2}$ and  $\xi_{d^1} (\bar x, z)$  is defined as in \eqref{varphi"}. Equality (1) follows from the definition in \eqref{2ndOrderDirectional} and the fact that $F$ and $f$ are continuously differentiable and $\varphi$ is directional differentiable at $\bar x$; cf. Theorem \ref{Theorem-phi'}. Equality (2) is based on the fact that $F$ and $f$ are twice continuously differentiable and $\varphi$ is second order directionally differentiable; cf. Theorem \ref{Theorem-phi''}. Equality (3) is based on \eqref{varphi"} and (4) is obtained thanks to the independence of $F''(\bar x, \bar y; d, e)$ and $f''(\bar x, \bar y; d, e)$ from the variable $z$. As for the final equality, (5), it results from the compactness of the set $S_1(\bar x; d)$, which is satisfied under the framework of Theorem \ref{Theorem-phi''}, see corresponding reference.

As for the second order directional derivative of $\phi''_\lambda$, let $d\in C^\lambda(\bar x, \bar y)$. Then $\phi'_\lambda(\bar x, \bar y; d)=0$, considering the fact that \eqref{DirectGreater} holds. Furthermore, as the point $\bar \zeta := (\bar x, \bar y, \bar z, \bar u, \bar v, \bar w)$ satisfies the optimality conditions \eqref{KS-1}--\eqref{KS-3}, it follows that
$$
\sum^{p}_{i=1}\bar u_i\nabla G_i(\bar x, \bar y)^\top d= \sum_{i\in \eta^1}\bar u_i \nabla G_i(\bar x, \bar y)^\top d + \sum_{i\in \theta^1}\bar u_i \nabla G_i(\bar x, \bar y)^\top d + \sum_{i\in \nu^1}\bar u_i \nabla G_i(\bar x, \bar y)^\top d = 0,
$$
considering the fact that $\bar u_i = 0$ for $i\in \eta^1 \cup \theta^1$ and $\nabla G_i(\bar x, \bar y)^\top d=0$ for $i\in \nu^1$ based on the fulfillment of \eqref{VS-3}, cf. partition in \eqref{multiplier sets}, and the definition of $C^\lambda(\bar x, \bar y)$, see \eqref{C(x,y)}. Similarly, we have $\sum^{q}_{j=1}\bar v_j\nabla g_j(\bar x, \bar y)^\top d=0$. Hence, from  \eqref{KS-1}--\eqref{VS-2}, we have
\begin{equation}\label{uopa}
\nabla F(\bar x, \bar y)^\top d + \lambda \nabla f(\bar x, \bar y)^\top d- \lambda \nabla_x\ell(\bar x, \bar z, \bar w)^\top d^1=0.
\end{equation}
Coming back to the definition of $C^\lambda(\bar x, \bar y)$, the last line in particular, it follows from \eqref{F'-lambda} that
$$
 \psi'_\lambda(\bar x, \bar y; d) = \max~\left\{\nabla F(\bar x, \bar y)^\top d + \lambda \nabla f(\bar x, \bar y)^\top d - \lambda\nabla_x\ell(\bar x, z, w)^\top d^1,\;\, z\in S(\bar x)\right\} = 0.
$$
Hence, from the expressions in \eqref{Psio2ndDirect} and \eqref{phi2ndDirectDolar}, it holds that for all $d\in C^\lambda(\bar x, \bar y)$ and $e\in \mathbb{R}^{n+m}$,
\begin{equation}\label{Psi2ndDirectQ}
\begin{array}{rl}
         \phi''_\lambda(\bar x, \bar y; d, e) =\max &\left\{\nabla G_i(\bar x, \bar y)^\top e + d^\top \nabla^2 G_i(\bar x, \bar y) d, \;\; i\in I^1(d),\right.\\
        & \;\; \nabla g_j(\bar x, \bar y)^\top e + d^\top \nabla^2 g_j(\bar x, \bar y) d,\;\; j\in I^2(d),\\
        &\;\; \left.(F+\lambda f)''(\bar x, \bar y; d, e)  - \lambda\nabla_1 \ell(\bar x, z, w)^\top e^1 - \lambda\xi_{d^1} (\bar x, z), \;  z\in S_1(\bar x; d^1) \right\}
\end{array}
\end{equation}
with $(F+\lambda f)''(\bar x, \bar y; d, e) = \nabla (F+\lambda f)(\bar x, \bar y)^\top e + d^\top \nabla^2 (F+\lambda f)(\bar x, \bar y) d$, as $F+\lambda f$ is twice continuously differentiable. The same can be said for any component of  $G$ or $g$.

It follows from \eqref{Psi2ndDirectQ} that for $d\in C^\lambda(\bar x, \bar y)\setminus \{0\}$, the optimization  problem in \eqref{molio}, i.e., to minimize the function $\phi_\lambda''(\bar x, \bar y; d, e)$ with respect to $e\in \mathbb{R}^{n+m}$, can be rewritten as
\begin{equation}\label{primalPB--1}
\underset{\varsigma}\inf~\left\{\varsigma_1|\;\, A\varsigma \geq b, \;\; \ubar{a}(z) \varsigma \geq \ubar{b}(z), \; z\in B\right\},
\end{equation}
where $A$, $b$, $\ubar{a}(z)$, $\ubar{b}(z)$, and $B$ are respectively defined by
$$
\begin{array}{l}
 A:=\left[\begin{array}{cc}
        1 & -\nabla G_1(\bar x, \bar y)^\top \\
         \vdots & \vdots\\
        1 & -\nabla G_{\iota_1}(\bar x, \bar y)^\top\\[1ex]
        1 & -\nabla g_1(\bar x, \bar y)^\top \\
        \vdots & \vdots\\
       1 &  -\nabla g_{\iota_2}(\bar x, \bar y)^\top
        \end{array}
 \right], \;\; b:=\left[\begin{array}{c}
         d^\top \nabla^2 G_1(\bar x, \bar y) d \\
         \vdots\\
         d^\top \nabla^2 G_{\iota_1}(\bar x, \bar y) d\\[1ex]
         d^\top \nabla^2 g_1(\bar x, \bar y) d \\
         \vdots\\
         d^\top \nabla^2 g_{\iota_2}(\bar x, \bar y) d
        \end{array}
 \right],\\[2ex]
\ubar{a}(z) := \Big[1, \; -\nabla F(\bar x, \bar y)^\top - \lambda \nabla f(\bar x, \bar y)^\top + \lambda\big[\nabla_1 \ell(\bar x, z, w)^\top, \; 0\big]\Big],
\end{array}
$$
 $\ubar{b}(z):=d^\top \nabla^2 F(\bar x, \bar y) d + \lambda d^\top\nabla^2 f(\bar x, \bar y) d - \lambda \xi_{d^1} (\bar x, z)$, and $B:=  S_1(\bar x; d^1)$.
Note that $\iota_1$ (resp. $\iota_2$) stands for the cardinality of $I^1(d)$ (resp. $I^2(d)$).
Clearly, problem \eqref{primalPB--1} is a semi-infinite optimization problem and for any vector $c\in \mathbb{R}\times \mathbb{R}^n \times \mathbb{R}^m$ such that $c_1 >0$ and $c_i:=0$ for $i=2, \ldots, n+m+1$,
$$
A^\top_i c > 0, \;\, i=1, \ldots, n+m \;\; \mbox{ and } \;\; \ubar{a}(z)^\top c>0 \mbox{ for all } \; z\in B,
$$
where $A_i$ represents row $i$ for $A$.
Hence, the extended Mangasarian-Fromowitz constraint qualification (see, e.g., \cite{hettich1993semi} for the definition) holds at any feasible point of problem \eqref{primalPB--1}. Combining this with the compactness of $S_1(\bar x; d^1)$, it follows from the duality theory of linear semi-infinite optimization (cf. latter reference)  that the dual of problem \eqref{primalPB--1} can be obtained as
\begin{equation}\label{ryou}
    \begin{array}{rl}
\underset{\kappa, u, v}\max & d^\top \nabla^2\bar{\mathcal{L}}^{\lambda}_{\kappa_\circ}(\bar x, \bar y, u, v) d     - \lambda\sum^{k}_{t=1}\kappa_t\xi_{d^1}(\bar x,  z^t)\\
\mbox{s.t. } &  (\kappa, u, v) \;\;\mbox{ satisfying } \eqref{onam1}-\eqref{onam3}
\end{array}
\end{equation}
for some $z^t\in S_1(\bar x; d_x)$, $t=1, \ldots, k$, with $k\in \mathbb{N}$ and $\kappa_{\circ}:= \sum^{\iota_1}_{i=1}\kappa_t$. Hence, condition \eqref{molio} holds if the one in \eqref{SOSCB} is satisfied for some $u$, $v$, $\kappa_t$, $z^t\in S_1(\bar x; d^1)$ with $\Lambda(\bar x, z^t)=\{w^t\}$, for $t=1, \ldots, k$, where $k\in \mathbb{N}$ and $\kappa_{\circ}:= \sum^{k}_{i=1}\kappa_t$, such that \eqref{onam1}--\eqref{onam3}.

Finally, \af{it} remains to show that ${[c]}$ holds; i.e., condition \eqref{2ndEpi} is satisfied. To proceed, first recall that under the assumptions of Theorem \ref{Theorem-phi''} that the negative of the optimal value function $\varphi$ \eqref{varphi} is second order epiregular (see \cite[Theorem 4.142]{BonnansShapiroBook2000}) at $\bar x$, i.e.,
\begin{equation}\label{val00}
\begin{array}{c}
-\varphi\left(\bar x + td^1 + \frac{1}{2}t^2 e^1(t)\right) \geq -\varphi (\bar x) - t \varphi' (\bar x; d^1) - \frac{1}{2} t^2 \varphi''\left(\bar x; d^1, e^1(t)\right) - \circ(t^2)
\end{array}
\end{equation}
 for any $d$, $t\geq 0$ and any function (path) $e$ from $\mathbb{R}_+$ to $\mathbb{R}^n \times \mathbb{R}^m$ (with $e(t):=(e^1(t), \; e^2(t))$, where $e^1(t)\in \mathbb{R}^n$) such that $te(t) \rightarrow 0$ as $t \downarrow 0$. Furthermore, as the function $(F+\lambda f)$ is twice continuously differentiable, it holds that
\begin{equation}\label{valo2}
\begin{array}{lll}
(F+\lambda f)\left((\bar x, \bar y) + td + \frac{1}{2}t^2 e(t)\right) & \geq & (F+\lambda f)(\bar x, \bar y) + t (F+\lambda f)'(\bar x, \bar y; d) \\
                                               & & \qquad \qquad + \frac{1}{2} t^2 (F+\lambda f)''\left(\bar x, \bar y; d, e(t)\right) + \circ(t^2).
\end{array}
\end{equation}
Multiplying \eqref{val00} by $\lambda$, which is positive, and adding the resulting inequality to \eqref{valo2},
\begin{equation}\label{valo3}
\begin{array}{c}
\left(F+\lambda (f -\varphi)\right)\left((\bar x, \bar y) + td + \frac{1}{2}t^2 e(t)\right) \;\; \geq \;\; \left(F+\lambda (f -\varphi)\right)(\bar x, \bar y) \\
                  \quad \qquad   \qquad \qquad + \;\; t \left(F+\lambda (f -\varphi)\right)'(\bar x, \bar y; d) + \frac{1}{2} t^2 \left(F+\lambda (f -\varphi)\right)''\left(\bar x, \bar y; d, e(t)\right) + \circ(t^2).
\end{array}
\end{equation}
Similarly to \eqref{valo2},  the upper- and lower-level constraint functions satisfy the second order epiregularity conditions. Subsequently, condition \eqref{2ndEpi} holds.
\end{proof}
Assuming that the lower-level optimal solution mapping $S$ \eqref{P} is single-valued at $\bar x$ can lead to a much simpler result, closely aligned to Theorem \ref{SOSSC-Theorem 1-1}, and subsequently to the derivation of a Robinson-type condition for problem \eqref{Penalized-LLVF}. To proceed,  we define $ C(\bar x, \bar y):= Q(\bar x, \bar y, \bar y)$; cf. \eqref{Ponana}.
\begin{thm}\label{SuffNon-NEW01}
Let $\bar\zeta:=(\bar x, \bar y, \bar z, \bar u, \bar v, \bar w)$ satisfy  \eqref{KS-1}--\eqref{KS-3}  for some $\lambda>0$, with $S(\bar x)=\{\bar y\}=\{\bar z\}$, and USCC (resp. LSCC) hold at $(\bar x, \bar y, \bar u)$ (resp. $(\bar x, \bar y, \bar v)$), i.e., $\theta^1 = \emptyset$ (resp. $\theta^2 = \emptyset$). Suppose that the lower-level problem is convex at $\bar x$  and the  assumptions in Theorem \ref{Theorem-phi''0} are satisfied  for all $d\in C(\bar x, \bar y)$. Then, $(\bar x, \bar y)$ is a strict local optimal solution of problem \eqref{Penalized-LLVF} provided that  the ULICQ \eqref{LICQ-leader} holds at $(\bar x, \bar y)$ and, for all $d\in C(\bar x, \bar y)\setminus \{0\}$, we have
\begin{equation}\label{SOSCB-11}
  \begin{array}{l}
d^\top \nabla^2\mathcal{L}^{\lambda}(\bar x, \bar y, \bar u, \bar v) d > \lambda\, \left(d^{1,3}\right)^\top\nabla^2\ell(\bar x,  \bar y, \bar w) d^{1,3}.
  \end{array}
  \end{equation}
\end{thm}
\begin{proof} Proceeding as in the proof of the previous theorem, we have the corresponding expressions of \eqref{melan} for $d\in \mathbb{R}^{n+m}$,  and \eqref{Psio2ndDirect} for  $d \in C(\bar x, \bar y)$ and $e\in \mathbb{R}^{n+m}$, given that $S_1(\bar x; d^1)\subseteq S(\bar x)=\{\bar y\}$; cf.  \eqref{S1(x,d)}. Further proceeding as in the previous theorem, the point $(\bar x, \bar y)$ is a strict local optimal solution of problem \eqref{Penalized-LLVF} if  for all  $d \in C(\bar x, \bar y)$, it holds that
\begin{equation}\label{SOSCB01-1}
  \begin{array}{l}
d^\top \nabla^2\bar{\mathcal{L}}^{\lambda}_{\kappa_\circ}(\bar x, \bar y, u, v) d > \lambda\, \kappa_\circ\xi_{d^1} (\bar x, \bar y)
  \end{array}
  \end{equation}
for some $(\kappa_\circ, u, v)$ verifying \eqref{onam3} together with the following conditions, while also taking into account the fact that $S_1(\bar x; d^1)\subseteq S(\bar x)=\{\bar y\}$ and $\{\bar w\} = \Lambda(\bar x, \bar y)$ (cf. Theorem \ref{Theorem-phi''0}):
\begin{eqnarray}\label{optimality conditions OK-1}
\nabla_1\bar{\mathcal{L}}^{\lambda}_{\kappa_\circ}(\bar x, \bar y, u, v) - \lambda \kappa_\circ \nabla_1 \ell(\bar x, \bar y, \bar w)=0, \;\;
 \nabla_2\bar{\mathcal{L}}^{\lambda}_{\kappa_\circ}(\bar x, \bar y, u, v)=0.\label{KS-11}
\end{eqnarray}
One can easily show that the existence of $(\kappa_\circ, u, v)$ verifying \eqref{KS-2} and \eqref{KS-3} (cf.  $\{\bar w\} = \Lambda(\bar x, \bar y)$), together with \eqref{onam3} and \eqref{KS-11}, such that \eqref{SOSCB01-1} holds, is equivalent to the existence of some $(\kappa'_\circ, u', v')\neq 0$ satisfying \eqref{KS-2}, \eqref{KS-3},  \eqref{KS-11}, and
 $$
 \kappa_\circ \geq 0, \;\; u_i\geq 0\,\mbox{ for }\, i\in I^1(d),\;\; v_j\geq 0 \,\mbox{ for }\, j\in I^2(d)
 $$
 such that \eqref{SOSCB01-1} holds, with $(\kappa_\circ, u, v)$  replaced by $(\kappa'_\circ, u', v')$. Subsequently, taking into account the fact that $I^1(d) \subseteq I^1$ and $I^2(d) \subseteq I^2$, it is not difficult to show that with the ULICQ satisfied at $(\bar x, \bar y)$, this point is a strict local optimal solution of problem \eqref{Penalized-LLVF} provided for all  $d \in C(\bar x, \bar y)$,
\begin{equation}\label{SOSCB01-1-2}
  \begin{array}{l}
d^\top \nabla^2\bar{\mathcal{L}}^{\lambda}_{1}(\bar x, \bar y, u, v) d > \lambda\,\xi_{d^1} (\bar x, \bar y)
  \end{array}
  \end{equation}
holds for some $(u, v)$ such that $(\bar x, \bar y, \bar z, u, v, \bar w)$ satisfies \eqref{KS-1}--\eqref{KS-3}, with $\bar z = \bar y$. It therefore remains to  show that we have $\bar u=u$ and $\bar v=v$.

To proceed, let us first show that $I^1(d)=I^1$ and $I^2(d)=I^2$ for any $d\in C(\bar x, \bar y)$. Obviously, $I^1(d)\subseteq I^1$ and $I^2(d) \subseteq I^2$, by definition. For the converse, let $i\in I^1$. Then $i\in \nu^1$, given that $\theta^1=\emptyset$. Furthermore, considering the definition of $C(\bar x, \bar y)$, it follows that $\nabla G_i(\bar x, \bar y)^\top d=0$. Hence, $i\in I^1(d)$. Subsequently, as $(\bar x, \bar y, \bar u, \bar v, \bar w)$ and $(\bar x, \bar y, u, v, \bar w)$ both satisfy \eqref{KS-1}--\eqref{KS-3} with $\bar z=\bar y$,
  $$
  \sum_{i\in I^1}(\bar u_i - u_i)\nabla G_i(\bar x, \bar y) + \sum_{j\in I^2}(\bar v_j - v_j)\nabla g_j(\bar x, \bar y)  =0.
  $$
Based on the ULICQ at $(\bar x, \bar y)$, it follows that $\bar u=u$ and $\bar v=v$, given that the components of these vectors are all zero when $i\notin I^1$ and $j\notin I^2$, respectively. To conclude the proof, observe that condition \eqref{SOSCB-11} is sufficient for \eqref{SOSCB01-1-2} to hold, considering the definition of  $\xi_{d^1} (\bar x, \bar y)$ in \eqref{varphi"-REST}.
\end{proof}

It is clear that under the framework of this theorem, the conclusion of Theorem \ref{SOSSC-Theorem 1-1} is valid, while also guarantying that the resulting point $(\bar x, \bar y)$ is a strict local optimal solution of problem \eqref{Penalized-LLVF}. Hence, Theorem \ref{SuffNon-NEW01} can be seen as an extension of the Robinson condition to bilevel optimization, in the sense discussed at the beginning of this section.

A common point between Theorems \ref{SuffNon-NEW} and \ref{SuffNon-NEW01} is that the \af{SSOSC}-type condition, i.e., \eqref{SOSCB} and \eqref{SOSCB-11}, respectively, are the main assumptions, as the remaining ones (except from the USCC and LSCC in Theorem \ref{SuffNon-NEW01}) are mostly technical, helping to ensure that the LLVF $\varphi$ \eqref{varphi} is second order directionally differentiable. Observe that the USCC and LSCC in Theorem \ref{SuffNon-NEW01} help to ensure that points satisfying \eqref{onam1}--\eqref{onam3} coincide with stationarity points in the sense \eqref{KS-1}--\eqref{KS-3}.

Next, we provide a small example illustrating Theorem \ref{SuffNon-NEW01}, with the corresponding Robinson-type framework.
\begin{example}
Considering the simple bilevel optimization problem
$$
\underset{x,y}\min~xy \;\,\mbox{ s.t. }\;\, x+y\leq 2, \;y\in S(x):=\underset{y}{\arg\min}~\left\{y~|\;\, x-y\leq 0\right\},
$$
we obviously have $S(x)=\{ x \}$ for all $x\in \mathbb{R}$ and one can easily check that the vector $(\bar x, \bar y, \bar z, \bar u, \bar v, \bar w)$ with $\bar x = \bar y = \bar z =0$, $\bar u=0$, $\bar v=\lambda $, and $\bar w=1$, satisfies the optimality conditions \eqref{KS-1}--\eqref{KS-3} for any $\lambda >0$. For this vector, $\theta^1=\nu^1 = \emptyset$, $\theta^2 = \theta^3=\emptyset$ and $\nu^2 =\eta^3 =\{1\}$. Hence, the critical cone $C(\bar x, \bar y, \bar z)=\left\{d\in \mathbb{R}^3|\; d_1=d_2=d_3 \right\}$, and subsequently, for any $d\in C(\bar x, \bar y, \bar z)\setminus \{0\}$,
$$
\left(\begin{array}{c}
         d_1\\
         d_2
       \end{array}
 \right)^\top \nabla^2 \mathcal{L}^\lambda (\bar \zeta) \left(\begin{array}{c}
         d_1\\
         d_2
       \end{array}
 \right) - \left(\begin{array}{c}
         d_1\\
         e_2
       \end{array}
 \right)^\top \nabla^2 \ell (\bar \zeta) \left(\begin{array}{c}
         d_1\\
         e_2
       \end{array}
 \right) = 2d^2_1 >0.
$$
All the other assumptions of Theorem \ref{SuffNon-NEW01} are also satisfied. $(0,0)$ is indeed the unique optimal solution of the problem above and the corresponding penalized version \eqref{Penalized-LLVF} for any $\lambda >0$.
\end{example}
Note that all the assumptions in Theorem  \ref{SuffNon-NEW01} also hold for the problem in Example \ref{example78}, given that at $(\bar x, \bar y, \bar v):=(0,0, 0, \lambda, \lambda)$ and $(\bar x, \bar z, \bar w):=(0,0, 0, 1, 1)$,
$
 \eta^2=\emptyset, \; \theta^2=\emptyset, \; \nu^2=\{1, \, 2\} \;\, \mbox{ and } \;\,  \eta^3=\emptyset, \; \theta^3=\emptyset, \; \nu^3=\{1, \, 2\},
$
 respectively. Hence, having $\nabla_2 g_i(\bar x, \bar z, \bar w)^\top\left(e_2, \,e_3\right)^\top=0$ for $i\in \nu^3$, is equivalent to $e_2 = e_3 =0$. Hence, combining this with the calculations in Example \ref{example78}, it is clear that all the assumptions in Theorem \ref{SuffNon-NEW01} hold for problem \eqref{example-pb}. 

For examples where the assumptions in Theorem \ref{SuffNon-NEW}, ensuring the second order differentiability of $\varphi$ \eqref{varphi}, hold without the uniqueness of the lower-level optimal solution, see, e.g., \cite{Shapiro1988,BonnansShapiroBook2000}.

To close this section, we provide an example confirming that the SSOSC-type condition \eqref{SOSCB-11} is essential in guarantying that a point is locally optimal for problem \eqref{Penalized-LLVF}, and potentially also a necessary condition.
\begin{example}
 Consider the following example from \cite[Chapter 5]{DempeFoundations2002}:
 $$
\underset{x,y}\min~\left(x-3.5\right)^2 + \left(y +4\right)^2 \;\,\mbox{ s.t. }\;\;y\in S(x):=\underset{y}{\arg\min}~\left\{(y-3)^2~|\;\, -x+y^2\leq 0\right\}.
$$
The lower-level solution set-valued mapping is single-valued; i.e., precisely,
$$
S(x)=\left\{\begin{array}{lll}
              \emptyset & \mbox{ if } & x<0,\\
               \{\sqrt{x}\} & \mbox{ if } & 0\leq x \leq 9,\\
                \{3\} & \mbox{ if } & x>9,
            \end{array}
\right.
$$
and we can easily check that the point $(1, 1, 1, 2\lambda -5, 2)$ satisfies the optimality conditions \eqref{KS-1}--\eqref{KS-3}, for $\lambda > \frac{5}{2}$. Obviously, at $(1, 1, 2\lambda -5)$ and $(1, 1, 2)$, we have
 $$
 \eta^2=\emptyset, \; \theta^2=\emptyset, \; \nu^2=\{1\} \;\, \mbox{ and } \;\,  \eta^3=\emptyset, \; \theta^3=\emptyset, \; \nu^3=\{1\},
 $$
respectively. Moreover, $C(1,1,1)=\left\{(d_1, d_2, e_2)|\; d_2 = e_2 = \frac{1}{2} d_1 \right\}$ and for all $(d_1, d_2, e_2)\in C(1,1,1)$,
 $$
\left(\begin{array}{c}
        d_1\\d_2
      \end{array}
\right)^\top \nabla^2 \mathcal{L}^\lambda \left(1, 1, 1, 2\lambda -5, 2\right) \left(\begin{array}{c}
        d_1\\d_2
      \end{array}
\right) - \left(\begin{array}{c}
        d_1\\e_2
      \end{array}
\right)^\top \nabla^2 \ell \left(1, 1, 2\right) \left(\begin{array}{c}
        d_1\\e_2
      \end{array}
\right)=0.
$$
We can however check that all the other assumptions of Theorem \ref{SuffNon-NEW01} are satisfied for $(1, 1, 1, 2\lambda -5, 2)$.
\end{example}
Further analysis on the necessity of condition \eqref{SOSCB-11} and the corresponding condition in Theorem \ref{SuffNon-NEW}, will be studied more closely in a future work. Different types of sufficient optimality conditions for bilevel optimization problems can be found in the papers \cite{DempeNecessary1992,DempeGadhiSecondOrder2010}.

\section{Numerical experiments}\label{Numerical experiments}

Based on our implementation of Algorithm \ref{algorithm 1} in \MATLAB (R2018a), we report and discuss test results obtained for the 124 nonlinear bilevel programs from the current version of the BOLIB library \cite{BOLIB2017}.

Recall that the necessary optimality conditions \eqref{KS-1} -- \eqref{KS-3} and their reformulation
\eqref{Eq-Main} as nonsmooth system of equations contain the penalization parameter $\lambda>0$.
Based on the construction process of the system (see, e.g., \cite{DempeDuttaMordukhovichNewNece, YeZhuOptCondForBilevel1995}), there is no specific rule on how to choose the  best value of this parameter. Rather, one may try all $\lambda$ from a certain finite discrete set in $(0,\infty)$, solve the corresponding optimality conditions, and then choose the best solution in terms of the upper-level objective function value.
For our approach, it turned out that a small set of $\lambda$-values  is sufficient to reach very good results.
To be precise, for all our experiments, we just used the nine values of $\lambda$ in $\Lambda:=\lambda\in\{2^{-1}, 2^{0}, \cdots, 2^{6}, 2^{7}\}$.

\subsection{Implementation details}

Besides the selection of penalization parameters described before, the other parameters needed in Algorithm \ref{algorithm 1} are set to
\[
\beta:=10^{-8},\quad \epsilon:=10^{-8},\quad t:=2.1,\quad \rho:=0.5,\quad \sigma:=10^{-4}.
\]
For each test example, we only used one starting point. The choices for starting points $x^o$ and $y^o$ are as follows.  If  an  example in literature comes with a starting point, then we use this point for our experiments. Otherwise, we choose $x^o=\textbf{1}_n, y^o=\textbf{1}_m$ except for three examples {\bf No} 20,119 and 120 because their global optimal solutions are $(\textbf{1}_n, \textbf{1}_m)$, where $\textbf{1}_n:=(1,\cdots,1)^\top\in\mathbb{R}^n$. So, for these three examples we used $x^o=-\textbf{1}_n, y^o=-\textbf{1}_m$.  Detailed information on starting points can be found in \cite{FischerZemkohoZhouDetailed2018}.   
Moreover, to fully define $\zeta^o=(x^o,y^o,z^o,u^o,v^o,w^o)$, we set
\[
z^o:=y^o,~u^o:=\left(|G_1(x^o,y^o)|, \cdots,|G_p(x^o,y^o)|\right)^\top,~
v^o:= \left(|g_1(x^o,y^o)|, \cdots, |g_q(x^o,y^o)|\right)^\top,~w^o:=v^o.
\]
In addition to the stopping criterion $\|\Phi^{\lambda}(\zeta^k)\|\le\epsilon$ used in Algorithm \ref{algorithm 1}, the algorithm is terminated if the iteration index $k$ reaches 2000. 

Finally,  to pick an element from the generalized B-subdifferential $\partial_B\Phi^{\lambda}(\zeta^k)$ in Step 2 of Algorithm \ref{algorithm 1}, we adopt the technique in \cite{DeLuca1996}.

\subsection{Test results}

Table \ref{numerical-results} lists, for any of the 124 test examples of the BOLIB library \cite{BOLIB2017}, values of the leader's objective function $F$.
The column  {\boldmath $F_{known}$} shows the best known $F$-values from the literature. Such a value was not available for 6 of the test problems. This is marked by ``unknown'' in the column \textbf{Status}. For 83 examples, the best known $F$-value is even optimal (with status labelled as ``optimal''). For the remaining 35 test problems, the known {\boldmath $F$}-value might not be optimal and its status is
just set to ``known''. The columns below {\boldmath $F_{new}$} show the $F$-values obtained by
Algorithm \ref{algorithm 1} for the nine  penalization  parameters $\lambda$.

Note that three examples contain a parameter that should be provided by the user. They are  examples \textbf{No} 14, 39, and 40.  The first one is associated with $\rho\geq 1$, which separates the problem into 4 cases: (i) $\rho=1$, (ii) $1<\rho<2$, (iii) $\rho=2$, and (iv) $\rho>2$. The results presented in Table \ref{numerical-results} correspond to case (i). For the other three cases, our method still produced the true global optimal solutions.  Example \textbf{No} 39 has a unique global optimal solution. The result presented in Table \ref{numerical-results} is for $c=0$. We also tested our method when $c=\pm1$, and obtained the unique optimal solutions as well. Example \textbf{No} 40 contains the parameter $M>1$, and the results presented in Table \ref{numerical-results} correspond to $M=1.5$.

Let us first note that evaluating the performance of an algorithm for the bilevel optimization problem \eqref{P} is a difficult task since the decision whether a computed point is (close
to) a global solution of \eqref{P} basically requires computing the LLVF $\varphi$. Therefore, instead of doing this, we suggest the following way of comparing our results for Algorithm \ref{algorithm 1} with the
results from literature known for the test problems. For an approximate solution $(x,y)$ obtained from Algorithm \ref{algorithm 1}, we first compute
\[
 \delta_F:=\frac{F(x,y)-F_{known}}{\max\{1,|F_{known}|\}},\quad\delta_f:=\frac{f(x,y)-f_{known}}{\max\{1,|f_{known}|\}},
\]
where $F_{known}$ as above is the best known $F$-value from literature and $f_{known}$ the lower-level function value which corresponds to $F_{known}$. Moreover, we set
\[
 \delta:=\left\{
 \begin{array}{ll}
 \max\{|\delta_F|,|\delta_f|\},\quad&\mbox{if  \textbf{Status} is optimal},\\
 \max\{\delta_F,\delta_f\},\quad&\mbox{otherwise}.
 \end{array}
 \right.
\]
In the latter case, $\delta$ can become negative. This means that both $F(x,y)$ and $f(x,y)$ are smaller than the values for the point with best $F$-value known from literature. In the last
column of Table \ref{numerical-results}, marked by $\delta_*$ we provide the smallest $\delta$-value among those obtained for all $\lambda\in\Lambda$ for the corresponding test problem.

\begin{landscape}
{\small
\begin{longtable}[H]{p{0.35cm}p{4.7cm}rrrrrrrrrrrr}
\hline
\textbf{No}&\textbf{Examples} &\textbf{Status}&${\boldmath F_{known}}$& \multicolumn{9}{c}{${\boldmath F_{new}}$}&$\delta_*$\\\cline{5-13}
 &&&&$\lambda=2^{-1}$&$\lambda=2^0$&$\lambda=2^1$&$\lambda=2^{2}$&$\lambda=2^{3}$&$\lambda=2^{4}$&$\lambda=2^{5}$&$\lambda=2^{6}$&$\lambda=2^{7}$\\\hline
1	&	\texttt{	AiyoshiShimizu1984Ex2	}&	optimal	&	5.00	&	8.25	&	1.00	&	10.31	&	15.16	&	14.61	&	14.80	&	39.75	&	14.95	&	4.97	&	0.01	\\\rowcolor[gray]{.9}
2	&	\texttt{	AllendeStill2013	}&	optimal	&	1.00	&	1.44	&	1.25	&	1.11	&	0.82	&	0.90	&	0.94	&	0.97	&	0.98	&	0.99	&	0.01	\\
3	&	\texttt{	AnEtal2009	}&	optimal	&	2251.6	&	2251.6	&	2251.6	&	2251.6	&	2251.6	&	2251.6	&	2826.5	&	2269.5	&	2251.6	&	4425.5	&	0.00	\\\rowcolor[gray]{.9}
4	&	\texttt{	Bard1988Ex1	}&	optimal	&	17.00	&	49.31	&	23.13	&	50.00	&	53.78	&	58.55	&	62.59	&	65.33	&	25.00	&	17.00	&	0.00	\\
5	&	\texttt{	Bard1988Ex2	}&	optimal	&	-6600.0	&	-6600.0	&	-6600.0	&	0.0	&	0.0	&	-6600.0	&	-5641.6	&	-5641.6	&	-6600.0	&	0.0	&	0.00	\\\rowcolor[gray]{.9}
6	&	\texttt{	Bard1988Ex3	}&	optimal	&	-12.68	&	-12.68	&	-14.51	&	-13.89	&	-13.40	&	-13.07	&	-12.89	&	-12.78	&	-12.73	&	-12.71	&	0.00	\\
7	&	\texttt{	Bard1991Ex1	}&	optimal	&	2.00	&	2.00	&	2.00	&	2.00	&	2.00	&	2.00	&	2.00	&	2.00	&	2.00	&	2.00	&	0.00	\\\rowcolor[gray]{.9}
8	&	\texttt{	BardBook1998	}&	optimal	&	0.00	&	0.00	&	0.00	&	0.00	&	0.00	&	100.00	&	11.11	&	11.11	&	11.11	&	11.11	&	0.00	\\
9	&	\texttt{	CalamaiVicente1994a	}&	optimal	&	0.00	&	0.00	&	0.00	&	0.00	&	0.00	&	0.00	&	0.00	&	0.00	&	0.00	&	0.00	&	0.00	\\\rowcolor[gray]{.9}
10	&	\texttt{	CalamaiVicente1994b	}&	optimal	&	0.31	&	0.38	&	0.34	&	1.06	&	1.22	&	1.18	&	1.46	&	0.54	&	1.54	&	0.31	&	0.00	\\
11	&	\texttt{	CalamaiVicente1994c	}&	optimal	&	0.31	&	0.37	&	0.56	&	1.00	&	1.22	&	0.48	&	0.52	&	1.30	&	0.31	&	1.30	&	0.00	\\\rowcolor[gray]{.9}
12	&	\texttt{	CalveteGale1999P1	}&	optimal	&	-29.20	&	-20.00	&	-20.00	&	-29.20	&	-29.20	&	-13.00	&	-13.00	&	-29.20	&	-23.00	&	-16.00	&	0.00	\\
13	&	\texttt{	ClarkWesterberg1990a	}&	optimal	&	5.00	&	9.00	&	9.00	&	5.00	&	9.00	&	13.00	&	7.97	&	9.77	&	9.86	&	9.93	&	0.00	\\\rowcolor[gray]{.9}
14	&	\texttt{	Colson2002BIPA1	}&	optimal	&	250.00	&	250.00	&	250.00	&	250.00	&	250.00	&	250.00	&	250.00	&	250.00	&	250.00	&	250.00	&	0.00	\\
15	&	\texttt{	Colson2002BIPA2	}&	known	&	17.00	&	49.31	&	23.13	&	50.00	&	53.78	&	58.55	&	62.59	&	65.33	&	66.95	&	17.00	&	0.00	\\\rowcolor[gray]{.9}
16	&	\texttt{	Colson2002BIPA3	}&	known	&	2.00	&	2.00	&	2.00	&	2.00	&	2.00	&	2.00	&	2.00	&	2.00	&	2.00	&	2.00	&	0.00	\\
17	&	\texttt{	Colson2002BIPA4	}&	known	&	88.79	&	89.98	&	88.93	&	88.79	&	88.79	&	88.79	&	88.79	&	87.31	&	88.02	&	88.39	&	0.00	\\\rowcolor[gray]{.9}
18	&	\texttt{	Colson2002BIPA5	}&	known	&	2.75	&	3.42	&	2.00	&	2.75	&	30.64	&	2.99	&	3.12	&	3.19	&	3.22	&	3.23	&	-0.27	\\
19	&	\texttt{	Dempe1992a	}&	unknown	&		&	-2.00	&	0.00	&	0.00	&	-0.25	&	-0.13	&	-0.06	&	-0.03	&	-0.02	&	-0.01	&		\\\rowcolor[gray]{.9}
20	&	\texttt{	Dempe1992b	}&	optimal	&	31.25	&	33.98	&	33.49	&	32.41	&	31.25	&	31.25	&	31.25	&	31.25	&	31.25	&	31.25	&	0.00	\\
21	&	\texttt{	DempeDutta2012Ex24	}&	optimal	&	0.00	&	0.00	&	0.00	&	0.00	&	0.00	&	0.00	&	0.00	&	0.00	&	276.17	&	0.00	&	0.00	\\\rowcolor[gray]{.9}
22	&	\texttt{	DempeDutta2012Ex31	}&	optimal	&	-1.00	&	-1.07	&	-0.66	&	-0.38	&	-0.21	&	-0.11	&	0.00	&	-0.03	&	-0.02	&	-0.01	&	0.07	\\
23	&	\texttt{	DempeFranke2011Ex41	}&	optimal	&	-1.00	&	-1.00	&	-1.00	&	-1.00	&	0.00	&	-1.00	&	-1.00	&	0.00	&	0.00	&	0.00	&	0.00	\\\rowcolor[gray]{.9}
24	&	\texttt{	DempeFranke2011Ex42	}&	optimal	&	5.00	&	-0.88	&	0.20	&	1.20	&	1.80	&	4.20	&	5.37	&	4.83	&	4.99	&	10.37	&	0.00	\\
25	&	\texttt{	DempeFranke2014Ex38	}&	optimal	&	2.13	&	-0.88	&	-0.50	&	1.00	&	2.13	&	2.15	&	2.13	&	2.12	&	4.00	&	2.12	&	0.00	\\\rowcolor[gray]{.9}
26	&	\texttt{	DempeEtal2012	}&	optimal	&	-1.00	&	-3.00	&	-1.00	&	-1.00	&	-1.00	&	0.00	&	-1.00	&	-1.00	&	-1.00	&	-1.00	&	0.00	\\
27	&	\texttt{	DempeLohse2011Ex31a	}&	optimal	&	-6.00	&	0.13	&	-5.50	&	-5.50	&	-5.50	&	-0.16	&	-0.04	&	0.18	&	0.33	&	0.41	&	0.00	\\\rowcolor[gray]{.9}
28	&	\texttt{	DempeLohse2011Ex31b	}&	optimal	&	-12.00	&	-12.00	&	-12.00	&	-12.00	&	-12.00	&	-12.00	&	-5.38	&	-5.47	&	0.35	&	0.42	&	0.00	\\
29	&	\texttt{	DeSilva1978	}&	optimal	&	-1.00	&	-0.56	&	-0.75	&	-0.89	&	-1.18	&	-1.10	&	-1.06	&	-1.03	&	-1.02	&	-1.01	&	0.01	\\\rowcolor[gray]{.9}
30	&	\texttt{	FalkLiu1995	}&	optimal	&	-2.20	&	-1.56	&	-2.00	&	-2.16	&	-2.22	&	-2.51	&	-2.38	&	-2.32	&	-2.28	&	-2.27	&	0.01	\\
31	&	\texttt{	FloudasEtal2013	}&	optimal	&	0.00	&	8.25	&	0.00	&	5.81	&	15.16	&	39.75	&	29.75	&	29.88	&	14.95	&	14.98	&	0.00	\\\rowcolor[gray]{.9}
32	&	\texttt{	FloudasZlobec1998	}&	optimal	&	1.00	&	100.00	&	1.00	&	100.00	&	100.00	&	0.92	&	0.97	&	100.00	&	0.99	&	1.00	&	0.00	\\
33	&	\texttt{	GumusFloudas2001Ex1	}&	optimal	&	2250.0	&	2250.9	&	2250.0	&	2250.0	&	2500.0	&	2250.0	&	2311.2	&	2308.7	&	2306.9	&	2500.0	&	0.00	\\\rowcolor[gray]{.9}
34	&	\texttt{	GumusFloudas2001Ex3	}&	optimal	&	-29.20	&	-20.00	&	-20.00	&	-29.20	&	-20.00	&	-13.00	&	-20.00	&	-13.00	&	-48.95	&	-13.00	&	0.00	\\
35	&	\texttt{	GumusFloudas2001Ex4	}&	optimal	&	9.00	&	9.00	&	9.00	&	9.00	&	9.00	&	9.00	&	7.97	&	8.46	&	8.73	&	8.86	&	0.00	\\\rowcolor[gray]{.9}
36	&	\texttt{	GumusFloudas2001Ex5	}&	optimal	&	0.19	&	0.19	&	0.19	&	0.19	&	2.18	&	0.19	&	-0.39	&	0.19	&	0.19	&	0.19	&	0.00	\\\hline
\textbf{No}&\textbf{Examples} &\textbf{Status}&${\boldmath F_{known}}$& \multicolumn{9}{c}{${\boldmath F_{new}}$}&$\delta_*$\\\cline{5-13}
 &&&&$\lambda=2^{-1}$&$\lambda=2^0$&$\lambda=2^1$&$\lambda=2^{2}$&$\lambda=2^{3}$&$\lambda=2^{4}$&$\lambda=2^{5}$&$\lambda=2^{6}$&$\lambda=2^{7}$\\\hline
 37	&	\texttt{	HatzEtal2013	}&	optimal	&	0.00	&	1.00	&	0.50	&	0.25	&	-0.13	&	1079.59	&	648.93	&	376.50	&	406.77	&	217.98	&	0.13	\\\rowcolor[gray]{.9}
38	&	\texttt{	HendersonQuandt1958	}&	known	&	-3266.7	&	-3243.6	&	-3200.0	&	-3239.7	&	-3260.5	&	-3265.2	&	-3266.3	&	-3266.6	&	-3266.6	&	-3275.2	&	0.00	\\
39	&	\texttt{	HenrionSurowiec2011	}&	optimal	&	0.00	&	0.00	&	0.00	&	0.00	&	0.00	&	0.00	&	0.00	&	0.00	&	0.00	&	0.00	&	0.00	\\\rowcolor[gray]{.9}
40	&	\texttt{	IshizukaAiyoshi1992a	}&	optimal	&	0.00	&	0.00	&	0.00	&	0.00	&	0.00	&	0.00	&	0.00	&	0.00	&	0.00	&	0.00	&	0.00	\\
41	&	\texttt{	KleniatiAdjiman2014Ex3	}&	optimal	&	-1.00	&	-1.33	&	-1.33	&	-0.59	&	-0.53	&	-0.90	&	-0.75	&	-0.97	&	0.99	&	-0.99	&	0.01	\\\rowcolor[gray]{.9}
42	&	\texttt{	KleniatiAdjiman2014Ex4	}&	known	&	-10.00	&	-4.00	&	-4.48	&	-5.00	&	-3.19	&	-1.16	&	-2.00	&	-2.00	&	-3.31	&	-6.68	&	0.52	\\
43	&	\texttt{	LamparSagrat2017Ex23	}&	optimal	&	-1.00	&	-1.00	&	-1.00	&	-1.00	&	-1.00	&	-1.00	&	-1.00	&	-1.00	&	-1.00	&	-1.00	&	0.00	\\\rowcolor[gray]{.9}
44	&	\texttt{	LamparSagrat2017Ex31	}&	optimal	&	1.00	&	1.00	&	1.00	&	1.00	&	1.00	&	1.00	&	1.00	&	1.00	&	1.00	&	1.00	&	0.00	\\
45	&	\texttt{	LamparSagrat2017Ex32	}&	optimal	&	0.50	&	0.63	&	0.56	&	0.32	&	0.40	&	0.44	&	0.47	&	0.48	&	0.49	&	0.50	&	0.00	\\\rowcolor[gray]{.9}
46	&	\texttt{	LamparSagrat2017Ex33	}&	optimal	&	0.50	&	0.50	&	0.50	&	0.50	&	0.50	&	0.50	&	0.50	&	0.50	&	0.50	&	0.50	&	0.00	\\
47	&	\texttt{	LamparSagrat2017Ex35	}&	optimal	&	0.80	&	1.25	&	0.80	&	1.00	&	1.25	&	1.00	&	0.80	&	1.00	&	0.80	&	1.25	&	0.00	\\\rowcolor[gray]{.9}
48	&	\texttt{	LucchettiEtal1987	}&	optimal	&	0.00	&	0.50	&	0.75	&	0.00	&	0.82	&	0.90	&	0.94	&	0.97	&	0.98	&	0.99	&	0.00	\\
49	&	\texttt{	LuDebSinha2016a	}&	known	&	1.14	&	1.14	&	1.14	&	1.52	&	1.14	&	1.14	&	2.24	&	1.17	&	2.24	&	2.24	&	0.00	\\\rowcolor[gray]{.9}
50	&	\texttt{	LuDebSinha2016b	}&	known	&	0.00	&	0.04	&	0.04	&	0.05	&	0.05	&	0.05	&	0.05	&	0.05	&	0.05	&	0.05	&	0.04	\\
51	&	\texttt{	LuDebSinha2016c	}&	known	&	1.12	&	1.58	&	1.64	&	1.60	&	1.62	&	1.12	&	1.64	&	1.64	&	1.64	&	1.12	&	0.00	\\\rowcolor[gray]{.9}
52	&	\texttt{	LuDebSinha2016d	}&	unknown	&	 	&	0.00	&	-0.29	&	20.00	&	-192.00	&	19.80	&	-18.83	&	-192.00	&	-34.84	&	-200.75	&		\\
53	&	\texttt{	LuDebSinha2016e	}&	unknown	&		&	11.79	&	7.27	&	2.09	&	5.44	&	2.78	&	28.65	&	15.95	&	13.18	&	7.24	&		\\\rowcolor[gray]{.9}
54	&	\texttt{	LuDebSinha2016f	}&	unknown	&		&	0.00	&	0.00	&	0.00	&	-18.63	&	20.00	&	-165.12	&	0.00	&	0.00	&	0.00	&		\\
55	&	\texttt{	MacalHurter1997	}&	optimal	&	81.33	&	81.85	&	81.46	&	81.26	&	81.30	&	81.31	&	81.32	&	81.32	&	81.33	&	81.33	&	0.00	\\\rowcolor[gray]{.9}
56	&	\texttt{	Mirrlees1999	}&	optimal	&	1.00	&	0.00	&	0.00	&	0.00	&	0.14	&	4.01	&	0.01	&	0.01	&	0.87	&	0.01	&	0.13	\\
57	&	\texttt{	MitsosBarton2006Ex38	}&	optimal	&	0.00	&	0.00	&	0.00	&	0.00	&	0.00	&	0.00	&	0.00	&	0.00	&	0.00	&	0.00	&	0.00	\\\rowcolor[gray]{.9}
58	&	\texttt{	MitsosBarton2006Ex39	}&	optimal	&	-1.00	&	-1.00	&	-1.00	&	-1.00	&	-1.00	&	-1.00	&	0.08	&	0.01	&	0.00	&	0.00	&	0.00	\\
59	&	\texttt{	MitsosBarton2006Ex310	}&	optimal	&	0.50	&	-0.82	&	-0.71	&	-0.63	&	0.42	&	-0.54	&	-0.09	&	-0.07	&	0.50	&	0.50	&	0.00	\\\rowcolor[gray]{.9}
60	&	\texttt{	MitsosBarton2006Ex311	}&	optimal	&	-0.80	&	-0.80	&	-0.80	&	0.51	&	1.00	&	-0.09	&	-0.80	&	-0.50	&	0.50	&	0.18	&	0.30	\\
61	&	\texttt{	MitsosBarton2006Ex312	}&	optimal	&	0.00	&	10.00	&	10.00	&	5.01	&	-0.26	&	-0.13	&	2.58	&	-0.03	&	-0.02	&	7.30	&	0.02	\\\rowcolor[gray]{.9}
62	&	\texttt{	MitsosBarton2006Ex313	}&	optimal	&	-1.00	&	-1.58	&	-1.00	&	-2.00	&	-2.00	&	0.03	&	0.87	&	0.19	&	0.00	&	0.23	&	0.00	\\
63	&	\texttt{	MitsosBarton2006Ex314	}&	optimal	&	0.25	&	0.01	&	0.02	&	0.04	&	0.05	&	0.06	&	0.06	&	0.06	&	0.06	&	0.06	&	0.19	\\\rowcolor[gray]{.9}
64	&	\texttt{	MitsosBarton2006Ex315	}&	optimal	&	0.00	&	-2.00	&	-2.00	&	-0.63	&	0.65	&	1.59	&	1.80	&	1.90	&	1.95	&	1.98	&	0.65	\\
65	&	\texttt{	MitsosBarton2006Ex316	}&	optimal	&	-2.00	&	-3.00	&	-3.00	&	-3.00	&	1.01	&	0.00	&	-2.06	&	0.28	&	-0.50	&	-3.00	&	0.03	\\\rowcolor[gray]{.9}
66	&	\texttt{	MitsosBarton2006Ex317	}&	optimal	&	0.19	&	0.00	&	0.00	&	0.03	&	0.06	&	0.13	&	0.00	&	0.00	&	0.00	&	0.24	&	0.05	\\
67	&	\texttt{	MitsosBarton2006Ex318	}&	optimal	&	-0.25	&	0.00	&	-1.00	&	0.75	&	0.56	&	0.36	&	0.21	&	0.11	&	-1.00	&	0.00	&	0.25	\\\rowcolor[gray]{.9}
68	&	\texttt{	MitsosBarton2006Ex319	}&	optimal	&	-0.26	&	0.00	&	0.00	&	0.00	&	-0.31	&	0.47	&	0.57	&	0.00	&	-0.26	&	-0.02	&	0.00	\\
69	&	\texttt{	MitsosBarton2006Ex320	}&	optimal	&	0.31	&	0.00	&	0.00	&	0.00	&	0.01	&	0.01	&	0.02	&	0.02	&	0.03	&	0.03	&	0.28	\\\rowcolor[gray]{.9}
70	&	\texttt{	MitsosBarton2006Ex321	}&	optimal	&	0.21	&	0.21	&	0.21	&	0.04	&	0.11	&	0.36	&	0.18	&	0.20	&	0.20	&	0.21	&	0.00	\\
71	&	\texttt{	MitsosBarton2006Ex322	}&	optimal	&	0.21	&	0.21	&	0.21	&	0.04	&	0.11	&	0.16	&	0.18	&	0.20	&	0.20	&	0.21	&	0.00	\\\rowcolor[gray]{.9}
72	&	\texttt{	MitsosBarton2006Ex323	}&	optimal	&	0.18	&	0.18	&	0.18	&	0.18	&	0.18	&	0.18	&	0.15	&	0.28	&	0.00	&	0.25	&	0.00	\\\hline
\textbf{No}&\textbf{Examples} &\textbf{Status}&${\boldmath F_{known}}$& \multicolumn{9}{c}{${\boldmath F_{new}}$}&$\delta_*$\\\cline{5-13}
 &&&&$\lambda=2^{-1}$&$\lambda=2^0$&$\lambda=2^1$&$\lambda=2^{2}$&$\lambda=2^{3}$&$\lambda=2^{4}$&$\lambda=2^{5}$&$\lambda=2^{6}$&$\lambda=2^{7}$\\\hline
73	&	\texttt{	MitsosBarton2006Ex324	}&	optimal	&	-1.76	&	-1.99	&	-1.75	&	-0.74	&	-1.71	&	-0.94	&	-0.29	&	-0.29	&	-1.76	&	-0.29	&	0.00	\\\rowcolor[gray]{.9}
74	&	\texttt{	MitsosBarton2006Ex325	}&	known	&	-1.00	&	0.00	&	0.00	&	0.00	&	0.00	&	0.00	&	0.00	&	0.00	&	0.00	&	0.00	&	1.00	\\
75	&	\texttt{	MitsosBarton2006Ex326	}&	optimal	&	-2.35	&	-2.00	&	-2.00	&	1.36	&	0.32	&	-1.00	&	-2.00	&	0.95	&	1.00	&	0.32	&	0.15	\\\rowcolor[gray]{.9}
76	&	\texttt{	MitsosBarton2006Ex327	}&	known	&	2.00	&	0.04	&	2.06	&	0.01	&	2.06	&	1.16	&	0.64	&	1.00	&	1.00	&	1.00	&	0.05	\\
77	&	\texttt{	MitsosBarton2006Ex328	}&	known	&	-10.00	&	-5.04	&	-3.70	&	-3.89	&	-2.86	&	-3.00	&	-1.01	&	-3.95	&	-5.00	&	-2.94	&	0.97	\\\rowcolor[gray]{.9}
78	&	\texttt{	MorganPatrone2006a	}&	optimal	&	-1.00	&	-1.00	&	-1.00	&	-1.00	&	0.50	&	-1.00	&	0.50	&	0.94	&	0.97	&	0.98	&	0.00	\\
79	&	\texttt{	MorganPatrone2006b	}&	optimal	&	-1.25	&	0.50	&	-1.25	&	-0.50	&	0.50	&	0.50	&	0.63	&	0.69	&	-0.50	&	-0.75	&	0.00	\\\rowcolor[gray]{.9}
80	&	\texttt{	MorganPatrone2006c	}&	optimal	&	-1.00	&	-1.00	&	1.00	&	1.00	&	1.00	&	1.00	&	1.00	&	0.75	&	0.75	&	1.00	&	0.00	\\
81	&	\texttt{	MuuQuy2003Ex1	}&	known	&	-2.08	&	-3.26	&	-2.88	&	-2.56	&	-2.34	&	-2.22	&	-2.15	&	-2.11	&	-2.10	&	-2.09	&	-0.55	\\\rowcolor[gray]{.9}
82	&	\texttt{	MuuQuy2003Ex2	}&	known	&	0.64	&	0.64	&	0.64	&	0.64	&	0.64	&	0.64	&	0.64	&	0.64	&	0.64	&	0.64	&	0.01	\\
83	&	\texttt{	NieWangYe2017Ex34	}&	optimal	&	2.00	&	2.00	&	2.00	&	2.00	&	2.00	&	2.00	&	3.89	&	3.80	&	2.00	&	4.33	&	0.00	\\\rowcolor[gray]{.9}
84	&	\texttt{	NieWangYe2017Ex52	}&	optimal	&	-1.71	&	0.06	&	-1.39	&	0.00	&	0.39	&	0.00	&	-0.24	&	0.00	&	-0.10	&	0.00	&	0.19	\\
85	&	\texttt{	NieWangYe2017Ex54	}&	optimal	&	-0.44	&	0.00	&	0.00	&	0.00	&	0.00	&	0.00	&	0.00	&	-0.15	&	-0.01	&	0.00	&	0.98	\\\rowcolor[gray]{.9}
86	&	\texttt{	NieWangYe2017Ex57	}&	known	&	-2.00	&	-2.05	&	-2.04	&	0.00	&	0.00	&	0.00	&	0.00	&	0.00	&	0.00	&	0.00	&	0.01	\\
87	&	\texttt{	NieWangYe2017Ex58	}&	known	&	-3.49	&	-3.53	&	0.00	&	0.00	&	-0.08	&	0.00	&	0.00	&	0.00	&	0.77	&	-0.91	&	0.02	\\\rowcolor[gray]{.9}
88	&	\texttt{	NieWangYe2017Ex61	}&	known	&	-1.02	&	0.19	&	-1.11	&	-1.03	&	0.39	&	-1.02	&	-0.18	&	0.00	&	0.00	&	0.00	&	0.00	\\
89	&	\texttt{	Outrata1990Ex1a	}&	known	&	-8.92	&	-10.67	&	-10.02	&	-9.54	&	-9.25	&	-9.09	&	-9.01	&	-8.96	&	0.00	&	-8.93	&	0.00	\\\rowcolor[gray]{.9}
90	&	\texttt{	Outrata1990Ex1b	}&	known	&	-7.56	&	-10.43	&	-9.46	&	-8.69	&	-8.19	&	-7.90	&	-7.74	&	-7.66	&	-7.62	&	0.00	&	0.01	\\
91	&	\texttt{	Outrata1990Ex1c	}&	known	&	-12.00	&	-12.00	&	-12.00	&	0.00	&	0.00	&	0.00	&	0.00	&	0.00	&	0.00	&	0.00	&	0.00	\\\rowcolor[gray]{.9}
92	&	\texttt{	Outrata1990Ex1d	}&	known	&	-3.60	&	-4.64	&	-3.33	&	-3.67	&	-3.64	&	-0.56	&	-0.47	&	-3.60	&	-3.60	&	-0.37	&	0.01	\\
93	&	\texttt{	Outrata1990Ex1e	}&	known	&	-3.15	&	-3.93	&	-3.93	&	-3.92	&	-3.92	&	-0.67	&	-3.92	&	-0.39	&	-0.34	&	-3.79	&	-0.10	\\\rowcolor[gray]{.9}
94	&	\texttt{	Outrata1990Ex2a	}&	known	&	0.50	&	0.50	&	0.50	&	0.50	&	0.50	&	0.50	&	0.50	&	0.50	&	0.50	&	0.50	&	0.00	\\
95	&	\texttt{	Outrata1990Ex2b	}&	known	&	0.50	&	0.50	&	0.50	&	0.50	&	0.50	&	0.50	&	0.50	&	0.50	&	0.50	&	0.50	&	0.00	\\\rowcolor[gray]{.9}
96	&	\texttt{	Outrata1990Ex2c	}&	known	&	1.86	&	1.03	&	1.34	&	1.57	&	1.70	&	2.35	&	5.82	&	1.84	&	1.85	&	5.49	&	0.00	\\
97	&	\texttt{	Outrata1990Ex2d	}&	known	&	0.92	&	0.85	&	0.85	&	0.85	&	0.85	&	0.85	&	0.85	&	0.85	&	9.76	&	0.85	&	-0.07	\\\rowcolor[gray]{.9}
98	&	\texttt{	Outrata1990Ex2e	}&	known	&	0.90	&	0.47	&	0.63	&	0.74	&	0.82	&	0.86	&	0.88	&	0.89	&	0.89	&	5.56	&	0.00	\\
99	&	\texttt{	Outrata1993Ex31	}&	known	&	1.56	&	2.57	&	2.05	&	1.45	&	1.50	&	1.53	&	1.55	&	8.95	&	1.56	&	1.56	&	0.00	\\\rowcolor[gray]{.9}
100	&	\texttt{	Outrata1993Ex32	}&	known	&	3.21	&	4.02	&	3.59	&	3.11	&	3.16	&	3.18	&	3.19	&	3.20	&	3.20	&	3.21	&	0.00	\\
101	&	\texttt{	Outrata1994Ex31	}&	known	&	3.21	&	4.02	&	3.59	&	3.11	&	3.16	&	3.18	&	3.19	&	3.20	&	3.20	&	8.01	&	0.00	\\\rowcolor[gray]{.9}
102	&	\texttt{	OutrataCervinka2009	}&	optimal	&	0.00	&	0.00	&	0.00	&	0.00	&	-0.13	&	-0.06	&	-0.03	&	-0.02	&	-0.01	&	0.00	&	0.00	\\
103	&	\texttt{	PaulaviciusEtal2017a	}&	optimal	&	0.25	&	0.00	&	0.78	&	0.31	&	0.11	&	0.00	&	1.59	&	1.00	&	1.00	&	1.94	&	0.09	\\\rowcolor[gray]{.9}
104	&	\texttt{	PaulaviciusEtal2017b	}&	optimal	&	-2.00	&	-2.00	&	-2.00	&	-2.00	&	-2.00	&	-2.00	&	-2.00	&	0.94	&	-0.22	&	0.96	&	0.00	\\
105	&	\texttt{	SahinCiric1998Ex2	}&	optimal	&	5.00	&	9.00	&	9.00	&	5.00	&	9.00	&	13.00	&	7.97	&	9.77	&	9.86	&	9.93	&	0.00	\\\rowcolor[gray]{.9}
106	&	\texttt{	ShimizuAiyoshi1981Ex1	}&	optimal	&	100.00	&	73.96	&	83.47	&	90.50	&	94.88	&	97.33	&	98.64	&	99.31	&	99.65	&	99.83	&	0.00	\\
107	&	\texttt{	ShimizuAiyoshi1981Ex2	}&	optimal	&	225.00	&	237.00	&	225.00	&	225.00	&	175.00	&	200.00	&	212.50	&	218.75	&	221.88	&	1155.9	&	0.00	\\\rowcolor[gray]{.9}
108	&	\texttt{	ShimizuEtal1997a	}&	unknown	&	 	&	25.00	&	23.13	&	16.89	&	25.00	&	58.55	&	16.89	&	16.89	&	66.95	&	65.34	&		\\\hline
\textbf{No}&\textbf{Examples} &\textbf{Status}&${\boldmath F_{known}}$& \multicolumn{9}{c}{${\boldmath F_{new}}$}&$\delta_*$\\\cline{5-13}
 &&&&$\lambda=2^{-1}$&$\lambda=2^0$&$\lambda=2^1$&$\lambda=2^{2}$&$\lambda=2^{3}$&$\lambda=2^{4}$&$\lambda=2^{5}$&$\lambda=2^{6}$&$\lambda=2^{7}$\\\hline
109	&	\texttt{	ShimizuEtal1997b	}&	optimal	&	2250.0	&	2369.0	&	2269.1	&	2250.0	&	2321.7	&	2250.0	&	2311.2	&	2308.6	&	2426.5	&	2305.8	&	0.00	\\\rowcolor[gray]{.9}
110	&	\texttt{	SinhaMaloDeb2014TP3	}&	known	&	-18.68	&	-9.63	&	-10.36	&	-10.36	&	-5.27	&	-8.63	&	-5.27	&	-9.63	&	-18.73	&	-18.71	&	0.00	\\
111	&	\texttt{	SinhaMaloDeb2014TP6	}&	known	&	-1.21	&	-1.21	&	-1.21	&	-1.21	&	-1.21	&	2.32	&	-1.21	&	1.52	&	-1.21	&	-1.21	&	0.00	\\\rowcolor[gray]{.9}
112	&	\texttt{	SinhaMaloDeb2014TP7	}&	known	&	-1.96	&	-1.98	&	-1.96	&	-1.98	&	-1.00	&	-1.98	&	0.00	&	0.00	&	0.00	&	0.00	&	0.00	\\
113	&	\texttt{	SinhaMaloDeb2014TP8	}&	optimal	&	0.00	&	50.36	&	48.58	&	45.65	&	0.00	&	0.00	&	1225.00	&	29.31	&	0.00	&	388.32	&	0.0	\\\rowcolor[gray]{.9}
114	&	\texttt{	SinhaMaloDeb2014TP9	}&	known	&	0.00	&	0.00	&	0.00	&	0.00	&	0.00	&	0.00	&	0.00	&	0.00	&	10.32	&	10.21	&	0.00	\\
115	&	\texttt{	SinhaMaloDeb2014TP10	}&	known	&	0.00	&	0.00	&	0.00	&	0.00	&	0.00	&	0.00	&	0.00	&	0.00	&	0.00	&	0.00	&	0.00	\\\rowcolor[gray]{.9}
116	&	\texttt{	TuyEtal2007	}&	optimal	&	22.50	&	24.73	&	24.48	&	24.01	&	25.00	&	25.00	&	25.00	&	25.00	&	22.50	&	22.50	&	0.00	\\
117	&	\texttt{	Vogel2012	}&	optimal	&	1.00	&	0.44	&	1.78	&	2.78	&	4.00	&	3.67	&	3.84	&	3.92	&	3.96	&	3.98	&	0.78	\\\rowcolor[gray]{.9}
118	&	\texttt{	WanWangLv2011	}&	optimal	&	10.62	&	12.59	&	13.78	&	11.25	&	11.25	&	10.94	&	36.25	&	7.50	&	17.00	&	10.63	&	0.00	\\
119	&	\texttt{	YeZhu2010Ex42	}&	optimal	&	1.00	&	1.09	&	1.05	&	8.78	&	6.11	&	5.53	&	5.39	&	5.19	&	5.09	&	6.11	&	0.05	\\\rowcolor[gray]{.9}
120	&	\texttt{	YeZhu2010Ex43	}&	optimal	&	1.25	&	22.99	&	25.25	&	12.03	&	10.50	&	1.00	&	9.38	&	9.19	&	9.09	&	9.05	&	0.20	\\
121	&	\texttt{	Yezza1996Ex31	}&	optimal	&	1.50	&	1.50	&	1.50	&	1.50	&	1.50	&	1.50	&	1.50	&	3.51	&	3.96	&	3.49	&	0.00	\\\rowcolor[gray]{.9}
122	&	\texttt{	Yezza1996Ex41	}&	optimal	&	0.50	&	0.72	&	0.62	&	4.00	&	4.00	&	4.00	&	4.00	&	4.00	&	4.00	&	4.00	&	0.20	\\
123	&	\texttt{	Zlobec2001a	}&	optimal	&	-1.00	&	0.00	&	0.50	&	-0.04	&	0.00	&	0.00	&	-1.00	&	0.00	&	0.00	&	-1.00	&	0.00	\\\rowcolor[gray]{.9}
124	&	\texttt{	Zlobec2001b	}&	unknown	&	 	&	0.00	&	0.00	&	0.00	&	1.00	&	1.00	&	1.00	&	1.00	&	1.00	&	1.00	&		\\\hline
${}$\\
\caption{Upper-level objective function values at the solution for different selections of  $\lambda\in\Lambda$. \label{numerical-results}}
\end{longtable}
}
\end{landscape}

\begin{table}[H]
{\renewcommand\baselinestretch{1.15}\selectfont
\begin{tabular}{lrrrrrrrrr}\\\hline
$\lambda$ &$2^{-1}$&$2^0$&$2^1$&$2^{2}$&$2^{3}$&$2^{4}$&$2^{5}$&$2^{6}$&$2^{7}$\\\hline
Number of failures	&	6	&	3	&	2	&	1	&	5	&	9	&	12	&	14	&	16	\\
$\alpha_K=1$	&	109	&	111	&	113	&	115	&	112	&	108	&	108	&	109	&	103	\\
$y^K\approx z^K$	&	85	&	83	&	75	&	70	&	65	&	64	&	68	&	82	&	96	\\
$v^K\approx w^K$	&	32	&	45	&	39	&	37	&	37	&	41	&	41	&	38	&	42	\\
Average iterations	&	152.3	&	84.3	&	129.1	&	154.3	&	194.6	&	288.9	&	357.4	&	375.9	&	451.3	\\
Average time	&	0.18	&	0.08	&	0.13	&	0.17	&	0.21	&	0.27	&	0.32	&	0.34	&	0.40	\\
\hline
${}$\\
\end{tabular}}
\caption{Performance of Algorithm  \ref{algorithm 1} for $\lambda\in\Lambda$.}\label{perf-tab:1}
\end{table}

We then report the performance of Algorithm  \ref{algorithm 1} for $\lambda\in\Lambda$. Let $K$ denote the iteration number, where  Algorithm  \ref{algorithm 1} is terminated. The first row in Table {\ref{perf-tab:1}} shows the number of examples for which Algorithm  \ref{algorithm 1} failed, i.e., it did not reach $\|\Phi^{\lambda}(\zeta^K)\|\le\epsilon$. Failures were observed only for a very small portion of the 124 examples.
The second row  in Table \ref{perf-tab:1} reports for how many examples the last step was a full Newton step, i.e., $\alpha_K=1$. This behavior could be observed for more than 100 examples regardless of $\lambda\in\Lambda$.

 We also want to consider whether $y^{K}\approx z^{K}$  occurs. As mentioned in Remark \ref{fallom},  $y=z$ generates another well-known type of stationary conditions, i.e., \eqref{VS-2}--\eqref{KS-3} and \eqref{KS-1*}.
Moreover, $w$ is a multiplier associated to the lower-level constraint function $g$ in connection to the lower-level problem, $v$ is associated to the same function $g$, but in connection to the upper-level problem.
The analogy in the complementarity systems \eqref{VS-4} and \eqref{KS-3} raises the question whether $v$ and $w$ coincide in some examples.
In Table \ref{perf-tab:1}, the fifth and sixth rows list  how many examples, for each $\lambda$, were solved with  $y^{K}\approx z^{K}$ or $v^{K}\approx w^{K}$.
The symbol ``$\approx$'' used here means that $  \|y^{K}-z^{K}\|/max(1,\|z^{K}\|)  \le 0.01$ and $\|w^{K}-v^{K}\|/max(1,\|v^{K}\|)  \le 0.01$  respectively. More information about the test runs and time of Algorithm \ref{algorithm 1} is provided by the last two rows of Table \ref{perf-tab:1}.  The average iterations and CPU time (in seconds) are less than 500 and 0.5 seconds, respectively, over all runs without failure. 



\begin{figure}[H]
\centering
    \includegraphics[width=1\linewidth]{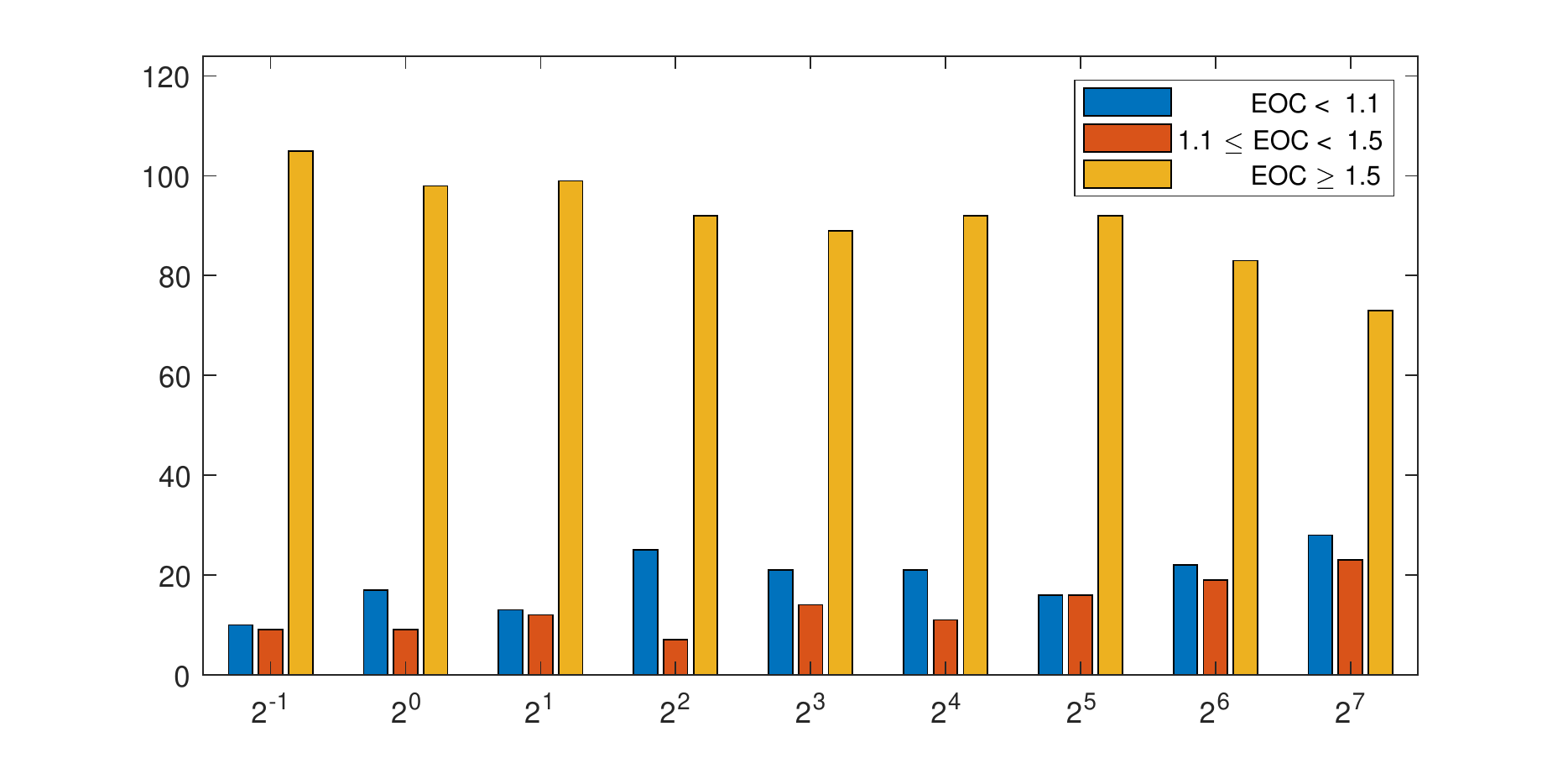}
  \caption{Experimental order of convergence (EOC)  of Algorithm  \ref{algorithm 1} for $\lambda\in\Lambda$.}
  \label{eoc}
\end{figure}

Finally, in order to  estimate the local behaviour of Algorithm \ref{algorithm 1} on  our test runs, Figure \ref{eoc} reports on the experimental order of convergence (EOC) defined by
$$
\mbox{EOC}:=\max\left\{\frac{\log\|\Phi^\lambda(\zeta^{K-1})\|}{\log\|\Phi^\lambda(\zeta^{K-2})\|},\, \frac{\log\|\Phi^\lambda(\zeta^{K})\|}{\log\|\Phi^\lambda(\zeta^{K-1})\|} \right\}.
$$
For example, if $\lambda=2^{-1}$, Algorithm \ref{algorithm 1} solved 10 examples with $\mbox{EOC}<1.1$, further 9 examples with $1\leq\mbox{EOC}<1.5$, and 105 examples with $\mbox{EOC}\geq 1.5$. Similar results were observed for other $\lambda\in\Lambda$.

More details on the numerical results discussed here can be found in the supplementary material \cite{FischerZemkohoZhouDetailed2018}. 

\bibliographystyle{plain}

\end{document}